\documentclass[11pt]{article}
\usepackage{multirow}
\usepackage{url}            
\usepackage{booktabs}       
\usepackage{amsfonts}       
\usepackage{nicefrac}       
\usepackage{microtype}      

\usepackage{amsmath}
\usepackage{amsthm}
\usepackage{amssymb}
\usepackage{amsfonts}
\usepackage{algorithm}  
\usepackage{algpseudocode}  
\usepackage{amsmath}  
\usepackage{graphicx}
\usepackage{subfigure}
\usepackage{bm}
\usepackage{color,xcolor}
\usepackage[colorlinks]{hyperref}
\usepackage{threeparttable}

\usepackage{booktabs}

\usepackage[top=1in, bottom=1in, left=1.25in, right=1.25in]{geometry}

\makeatletter
\@addtoreset{equation}{section}
\makeatother

\newtheorem{theorem}{Theorem}[section]
\newtheorem{lemma}{Lemma}[section]

\newtheorem{remark}{Remark}[section]

\newcommand{\inp}[1]{\left\langle#1\right\rangle}
\newcommand{\norm}[1]{\left\Vert#1\right\Vert}

\def\bbr{{\Bbb{R}}} 

\author{Bin Zhou \footnote{
			School of Science, Nanjing University of Posts and Telecommunications/Key Laboratory of NSLSCS, Ministry of Education,
		Nanjing, 210023, China, 415281@njupt.edu.cn} \and 
	Liusheng Hou
	\footnote{
		School of Information Engineering, Nanjing Xiaozhuang University, Nanjing 210023, China, houlsheng@njxzc.edu.cn} \and 
	Xingju Cai
	\footnote{Key Laboratory of NSLSCS, Ministry of Education, School of Mathematical Sciences, Nanjing Normal University, Nanjing 210023, China, caixingju@njnu.edu.cn} \and
	Hailin Sun
	\footnote{Corresponding author, 
		Key Laboratory of NSLSCS, Ministry of Education,  School of Mathematical Sciences,  Nanjing Normal University, Nanjing 210023, China, hlsun@njnu.edu.cn}}
\title{The G\"uler-type acceleration for proximal gradient, linearized augmented Lagrangian and  linearized alternating direction method of multipliers}
\usepackage[top=1in, bottom=1in, left=1.25in, right=1.25in]{geometry}

\begin{document}
	\maketitle
	
	\begin{abstract}
		In this paper, we introduce the G\"uler-type acceleration technique and utilize it to  propose three acceleration algorithms: the G\"uler-type accelerated proximal gradient method (GPGM), the G\"uler-type accelerated  linearized augmented Lagrangian method (GLALM) and the G\"uler-type accelerated linearized alternating direction method of multipliers (GLADMM). The key idea behind these algorithms is to fully leverage  the information of negative term \bm{$-\norm{x^k-\hat{x}^{k-1}}^2$} in order to design the extrapolation step. This concept of using negative terms to improve acceleration can be extended to other algorithms as well. Moreover, the proposed GLALM and GLADMM enable simultaneous acceleration of both  primal and dual variables. Additionally, GPGM and GLALM achieve the same convergence rate of $O(\frac{1}{k^2})$ with some existing results. Although GLADMM achieves the same total convergence rate of $O(\frac{1}{N})$ as in  existing results, the partial convergence rate is improved from $O(\frac{1}{N^{3/2}})$ to $O(\frac{1}{N^2})$. To validate the effectiveness of our algorithms, we conduct numerical experiments on various problem instances, including the $\ell_1$ regularized logistic regression, quadratic programming, and compressive sensing. The experimental results  indicate that our algorithms outperform existing methods in terms of efficiency.  This also demonstrates the potential of the stochastic algorithmic versions of these algorithms in application areas such as statistics, machine learning, and data mining. Finally, it is worth noting that this paper aims to introduce how Güler's acceleration technique can be applied to gradient-based algorithms and to provide a unified and concise framework for their construction. 
	\end{abstract}
	
\textbf{Keywords:}\quad G\"uler-type acceleration; Convergence rate; Proximal gradient; Linearized augmented Lagrangian;  Linearized alternating direction method \newline 

\textbf{AMS subject classifications:}\; 90C25; 90C30; 68W40 \newline

	\section{Introduction}
		Nesterov's acceleration, proposed by Nesterov in 1983 \cite{nesterov1983method}, is a renowned accelerated gradient method utilized for solving convex smooth unconstrained optimization problems. It has gained significant popularity due to its ability to expedite convergence, reduce the required number of iterations to achieve a solution, and enhance overall optimization performance. Inspired by  Nesterov \cite{nesterov1983method}, G\"uler \cite{guler1992new}  proposed an acceleration for proximal point algorithm (PPA), which is commonly referred to as G\"uler's acceleration.
	Note that the Nesterov's acceleration can also apply to PPA directly, and our idea of the G\"uler-type acceleration is coming from the difference between the two kinds of acceleration techniques.
	
	To explain our idea, we first compare two kinds of acceleration techniques on PPA. Consider the following convex optimization problem:
	\begin{equation}\label{COP}
	\underset{x \in X}{\min} ~g(x),
	\end{equation}
	where  $g:\mathbb{R}^n \rightarrow \mathbb{R}$ is a proper closed convex function, and $X \subseteq \bbr^n$ is a closed convex set. To solve problem \eqref{COP}, the classical proximal point algorithm (PPA) has been well studied in the literature under various contexts and frameworks, including various modifications (see e.g., \cite{he2009proximal,guler1992new,kim2021accelerated,guo2020nonsymmetric,rockafellar1976augmented,cai2022developments}). Nesterov's acceleration technique \cite{nesterov1983method,nesterov2007dual} can be applied to accelerate the PPA in the following form:
	\begin{equation*}
	\begin{aligned}
	&\hat{x}^k=x^k+\frac{t_{k-1}-1}{t_k}\left(x^k-x^{k-1}\right),\\
	&x^{k+1}=\underset{x \in X}{\arg\min}\left\{g(x)+\frac{1}{2\mu}\norm{x-\hat{x}^k}^2\right\},\\
	&t_{k+1}=\frac{1+\sqrt{1+4t^2_{k}}}{2}\notag,
	\end{aligned}
	\end{equation*}
	while the formulation of  G\"uler's acceleration \cite{guler1992new} on PPA is
	\begin{equation*}
	\begin{aligned}
	&\hat{x}^k=x^k+\frac{t_{k-1}-1}{t_k}\left(x^k-x^{k-1}\right)+\frac{t_{k-1}}{t_k}\left(x^k-\hat{x}^{k-1}\right),\\
	&x^{k+1}=\underset{x \in X}{\arg\min}\left\{g(x)+\frac{1}{2\mu}\norm{x-\hat{x}^k}^2\right\},\\
	&t_{k+1}=\frac{1+\sqrt{1+4t^2_{k}}}{2}\notag.
	\end{aligned}
	\end{equation*}
	G\"uler's acceleration,  in comparison to Nesterov's acceleration, has the advantage of fully utilizing the information of $-\norm{x^k-\hat{x}^{k-1}}^2$ for extrapolation design, which we named G\"uler-type acceleration technique, while Nesterov's acceleration overlooks this information. This may be the reason why we do not frequently use Nesterov's acceleration in PPA.
	
	Moreover,  Nesterov's acceleration has found widespread application in various algorithms, such as the proximal gradient method (PGM)\cite{beck2009fast}, linearized augmented Lagrangian method (LALM)\cite{xu2017accelerated}, and linearized alternating direction method of multipliers (L-ADMM)\cite{ouyang2015accelerated}, within important domains like statistics\cite{wen2017linear,vaswani2019fast}, machine learning\cite{chen2014optimal,lin2020accelerated}, and data mining\cite{sreedharan2024nesterov,wu2022fast}. Drawing on G\"uler's acceleration scheme, Kim and Fessler \cite{kim2016optimized} proposed two optimized first-order methods that achieve an $O(\frac{1}{k^2})$ convergence rate for smooth and unconstrained convex minimization. More recently, Jang et al. \cite{jang2025computer} introduced a double-function, step-size-optimization PEP framework for composite problems. 
	These motivate us to utilize the information provided by $-\norm{x^k-\hat{x}^{k-1}}^2$, 
	and explore the potential of G\"uler-type acceleration technique in various gradient-type algorithms.

	The PGM is a commonly used approach for solving the optimization problem:
	\begin{equation}\label{FGP}
	\min_{x \in X} \; F(x):=f(x)+g(x),
	\end{equation}
	where $X \subseteq \bbr^n$ is a closed convex set, $f:\mathbb{R}^n \rightarrow \mathbb{R}$ is smooth convex function such that, for some $L\geq 0$, 
	\begin{equation}\label{PG_L}
	f(x')\leq f(x)+\inp{\nabla f(x),x'-x}+\frac{L}{2}\norm{x'-x}^2,
	\end{equation}
	$g:\mathbb{R}^n \rightarrow \mathbb{R}$ is a relatively simple, proper, convex, lower semicontinuous function. (refer to    \cite{bauschke2011convex,combettes2005signal,daubechies2004iterative} for details).   The fast iterative shrinkage-thresholding algorithm (FISTA), proposed by Beck and Teboulle \cite{beck2009fast}, is a popular acceleration technique for PGM. The second acceleration PGM (APGM) \cite[Section 3.1]{tseng2010approximation} corresponding to Auslender and Teboulle‘s extension \cite{auslender2006interior} of Nesterov's second method \cite{nesterov1988approach}
	(i.e. Nesterov's second APGM). They utilize Nesterov's extrapolation techniques and achieves a faster convergence rate of $O\left(\frac{1}{k^2}\right)$. 
	Our first aim is to utilize the information provided by $-\norm{x^k-\hat{x}^{k-1}}^2$ in \eqref{GPGeq2}, design a G\"uler-type accelerated PGM (GPGM, Algorithm \ref{GPG}) and compare its performance with Nesterov's second APGM. 
	
	LALM  can be applied to solve linearly constrained composite convex programming as follows \cite{xu2017accelerated}:
	\begin{equation}\label{LCP}
	\begin{aligned} 
	\min_{x\in X} F(x):=f(x)+g(x)\text {, s.t. } Ax=b,
	\end{aligned}
	\end{equation}
	where $A\in \bbr^{l\times n}$, $b\in \bbr^l$, the setting of $X$, $f$ and $g$ are same as in \eqref{FGP}. Several accelerated LALM methods based on Nesterov's acceleration have been proposed for solving problem \eqref{LCP} \cite{xu2017accelerated,boct2023fast,arjevani2020ideal,hassan2019augmented}.   In contrast to existing accelerated LALM, we propose a G\"uler-type accelerated LALM (GLALM, Algorithm \ref{GALALM}). This approach fully utilizes the information of $-\|x^k-\hat{x}^{k-1}\|^2$ and $-\|z^k-\hat{z}^{k-1}\|^2$ in \eqref{GALeq3} and incorporates G\"uler-type acceleration for both primal and dual variables. 
	
	We also consider the linearly constrained two-block structured problems, which can be solved using the L-ADMM \cite{ouyang2015accelerated}. The problem is defined as follows:
	\begin{equation}\label{AECCO}
	\min _{x \in X, y \in Y} F(x,y)=f(x)+g(y) \text {, s.t. } B y-A x=b,
	\end{equation}
	where  $Y\subset \bbr^m$ is closed convex set, $B\in \bbr^{l\times m}$, and the setting of $A$, $b$, $X$, $f$ and $g$ are same as in \eqref{LCP}. Nesterov's acceleration has been used to accelerate L-ADMM \cite{xu2017accelerated,ouyang2015accelerated,li2019accelerated,kadkhodaie2015accelerated}. 
	Similar to GLALM, we fully utilize the information of $-\|x^k-\hat{x}^{k-1}\|^2$, $-\|y^k-\hat{y}^{k-1}\|^2$, and $-\|z^k-\hat{z}^{k-1}\|^2$ in \eqref{GADMMeq6} to design the G\"uler-type acceleration for both primal and dual variables, and propose a G\"uler-type accelerated L-ADMM (GLADMM, Algorithm \ref{GADMM}) to solve problem \eqref{AECCO}.

	The main contributions of this paper can be summarized as follows.
	\begin{itemize}
		\item We introduce the G\"uler-type acceleration technique and utilize it to propose three acceleration algorithms, namely GPGM, GLALM, and GLADMM. The key idea is to expedite the algorithms by fully utilizing the information from negative terms ($-\| \hat{x}^k - x^{k+1} \|^2$, $-\| \hat{y}^k - y^{k+1} \|^2$ and $-\| \hat{z}^k - z^{k+1} \|^2$) in GPGM (Algorithm \ref{GPG}), GLALM (Algorithm \ref{GALALM}), and GLADMM (Algorithm \ref{GADMM}).
		
		\item With G\"uler-type acceleration techniques, both GPGM and GLALM achieve the same convergence rate as Nesterov's second APGM and  accelerated linearized augmented Lagrangian method (ALALM, \cite[Algorithm 1]{xu2017accelerated}) while providing a broader range of coefficient options. 
		Additionally, GLADMM achieves the same total convergence rate of $O(\frac{1}{N})$ as in accelerated L-ADMM (AL-ADMM)  \cite{ouyang2015accelerated}, but with an improved partial convergence rate from $O(\frac{1}{N^{3/2}})$ to $O(\frac{1}{N^2})$.

		\item Compared to existing algorithms such as ALALM and AL-ADMM  \cite{ouyang2015accelerated}, the proposed GLALM and GLADMM can simultaneously accelerate both the primal and dual variables.
	\end{itemize}

	\begin{center}
		\resizebox{\textwidth}{!}{
	\begin{threeparttable}[htbp]
		\label{Table 0}
		\caption{Comparing the convergence rates of GPG and second Nesterov's PG, GLALM and ALALM, GLADMM and AL-ADMM, respectively.}
		\belowrulesep=0pt\aboverulesep=0pt
		\begin{tabular}{ccccccc}
			\toprule\toprule
			Algorithm & GPG (Algorithm \ref{GPG}) &  second Nesterov's PG\cite{tseng2010approximation} \\ 
			\hline
			Convergence rate        &$\frac{2\gamma}{2-\alpha}\frac{\norm{x^1-x^*}^2}{(k+1)^2}$, $\alpha\in(0,1)$ &$\frac{\norm{x^1-x^*}^2}{(k+1)^2}$  \\
			\hline
			Algorithm & GLALM (Algorithm \ref{GALALM}) & ALALM\cite{xu2017accelerated} \\
			\hline
			Convergence rate & $\frac{1}{k(k+1)}\left(\frac{\eta}{2-\alpha}\left\|x^1-x^*\right\|^2+\frac{\max \left\{\left(1+\left\|z^*\right\|\right)^2, 4\left\|z^*\right\|^2\right\}}{\gamma\kappa}\right)$, $\alpha\in(0,1)$, $\kappa\in (1,2)$  &$\frac{1}{k(k+1)}\left(\eta\left\|x^1-x^*\right\|^2+\frac{\max \left\{\left(1+\left\|z^*\right\|\right)^2, 4\left\|z^*\right\|^2\right\}}{\gamma}\right)$  \\
			\hline
			Algorithm & GLADMM (Algorithm \ref{GADMM}) & AL-ADMM\cite{ouyang2015accelerated}\\
			\hline
			Convergence rate  & $O(\frac{1}{N^2})+O(\frac{1}{N})+O(\frac{1}{N})$ &$O(\frac{1}{N^{3/2}})+O(\frac{1}{N})+O(\frac{1}{N})$ \\
			\bottomrule
			\bottomrule
		\end{tabular}
	\end{threeparttable}
}
\end{center}
	
	The remainder of the paper is structured as follows:  In Sections 2-4, we introduce the GPGM, GLALM, and GLADMM algorithms along with their convergence analysis. Section 5 presents the results of numerical experiments. Finally, we conclude the paper in Section 6. 
	
	\section{Notation and Preliminaries}
	Throughout this paper, we adopt the following standard notation. Let $\mathbb{R}^{n}$ denote an $n$-dimensional Euclidean space.
	For vectors $x,y\in \mathbb{R}^{n}$, $\inp{x,y}$ and $\norm{x}=\sqrt{\inp{x,x}}$ denote the standard inner product and the Euclidean norm, respectively. 
	For any $x \in \mathbb{R}$, we define $[x]_{+}$ as $\max \{x, 0\}$. 
	We denote the gradient of any convex function $f(x)$ by $\nabla f(x)$. The diameter of a bounded set $X$ is denoted by $\Omega_X$
	(i.e. $\Omega_X:=\sup_{x,x'\in X}\norm{x-x'}$). For brevity, we use “w.r.t.” as an abbreviation for “with respect to”.
	
To support our convergence analysis, we introduce several key lemmas below. These results will be frequently referenced in the sequel when analyzing the convergence behavior of the proposed algorithms.  
	\begin{lemma}[{\cite[Chapter 5]{bauschke2011convex}}]\label{fact}
		Let  $a,b,c,d\in \bbr^n$, and $t\in \bbr$. Then
		\begin{itemize}
			\item[(i)] $(2-t)\left(\norm{a+b}^2-\norm{a}^2-(1-t)\norm{b}^2\right)=\norm{a+b}^2-\norm{a-(1-t)b}^2$;
			\item[(ii)] $t\left(\norm{a+b}^2-\norm{a}^2-(t-1)\norm{b}^2\right)=\norm{a+b}^2-\norm{a-(t-1)b}^2$;
			\item[(iii)] $t(t-1)\norm{a}^2+t\norm{b}^2=\norm{(t-1)a+b}^2+(t-1)\norm{a-b}^2$;
			\item[(iv)] $2\inp{a-b,c-d}=\norm{a-d}^2-\norm{a-c}^2+\norm{b-c}^2-\norm{b-d}^2$.
		\end{itemize}
	\end{lemma}

	\begin{lemma}\label{factt} 
		If $t_0=0$ and $t_{k+1}=\cfrac{1+\sqrt{1+4t^2_{k}}}{2}$ for $k\geq 1$, we have 
		\begin{equation*}
		t_{k-1}^2=t_k(t_k-1) ~~~\text{and}~~~t_k\geq \frac{k+1}{2},~~~\forall k\geq 1. 
		\end{equation*}
	\end{lemma}
This sequence $\{t_k\}$ is commonly used in Nesterov-type acceleration schemes and plays a crucial role in achieving optimal convergence rates. 
The following lemmas provide essential tools for bounding both the optimality gap and the feasibility residual in constrained optimization problems.	
	\begin{lemma}[{\cite[Lemma 2.1]{xu2017accelerated}}]\label{rate1}
		Given a function $\phi$ and a fixed point $\bar{x}$, if for any $z$ it holds that
		\[
		F(\bar{x})-F\left(x^*\right)-\inp{z, A \bar{x}-b} \leq \phi(z),
		\]
		then for any $\rho>0$ we have
		$
		F(\bar{x})-F(x^*)+\rho\norm{A \bar{x}-b} \leq \sup _{\norm{z} \leq \rho} \phi(z). 
		$
	\end{lemma}
	\begin{lemma}[{\cite[Lemma 2.2]{xu2017accelerated}}]\label{rate2}
		For any $\epsilon \geq 0$, if
		$
		F(\bar{x})-F(x^*)+\rho\norm{A \bar{x}-b} \leq \epsilon, 
		$
		then we have, for any $\rho>\norm{z^*}$,  
		\[
		\norm{A \bar{x}-b} \leq \frac{\epsilon}{\rho-\norm{z^*}}~~\text{and}~~-\frac{\norm{z^*} \epsilon}{\rho-\norm{z^*}} \leq F(\bar{x})-F(x^*)\leq \epsilon. 
		\]
	\end{lemma}
With these essential tools at hand, we proceed to analyze the convergence behavior of the proposed GPGM, GLALM, and GLADMM algorithms in the subsequent sections.

	\section{The G\"uler-type accelerated PGM}
	In this section, we utilize the G\"uler acceleration technique to develop the GPGM for solving \eqref{FGP}. The proposed GPGM is as in Algorithm \ref{GPG}. 
	\begin{algorithm}[h]  
		\caption{G\"uler-type accelerated proximal gradient method (GPGM)}  
		\label{GPG}  
		\begin{algorithmic}[1]  
			\Require  initial point $\hat{x}^1=x^1\in\bbr^n$, $\alpha=\cfrac{L}{\gamma}$, $\gamma$ is a constant, $t_0 = 0$ and $t_1=1$
			\State $k \leftarrow 1$  
			\Repeat  
			\State Generate
			\begin{align}
			&t_{k}=\frac{1+\sqrt{1+4t^2_{k-1}}}{2},\theta_k=\frac{1}{t_{k}},\tau_k=\gamma \theta_k=\frac{L\theta_k}{\alpha}\label{PGpara}\\
			&x^k_{md}=(1-\theta_k)x^k_{ag}+\theta_k \hat{x}^k\label{PGsxd}\\
			&x^{k+1}=\underset{x \in \bbr^n}{\arg\min}\left\{g(x)+\frac{\tau_k}{2}\norm{x-\left(\hat{x}^k-\frac{1}{\tau_k}\nabla f(x^k_{md})\right)}^2\right\}\label{PGxk}\\
			&\bm{\hat{x}^{k+1}=(\alpha-1)\hat{x}^{k}+(2-\alpha)x^{k+1}}\label{PGhatx}\\
			&x^{k+1}_{ag}=(1-\theta_k)x^k_{ag}+\theta_k x^{k+1}\label{PGxag}		
			\end{align}
			\State $k \leftarrow k+1$
			\Until $x^{k+1}_{ag}$ satisfies the stop criterion.
			\Ensure  $x^{k+1}_{ag}$  
		\end{algorithmic}  
	\end{algorithm}

    Note that if, in Algorithm \ref{GPG}, $\alpha=1$ and $\gamma=L$, then $\hat{x}^k=x^k$, and Algorithm \ref{GPG} is the same as the
    Nesterov's second APGM \cite[Section 3.1]{tseng2010approximation}. Furthermore, if $g(x)=0$, it reduces to a version of Nesterov’s accelerated gradient method in \cite{nesterov1988approach}. 
    In what follows, we consider the convergence analysis of the proposed GPGM. 
    \begin{theorem}\label{ThGPG}
    	Let $x^*$ denote a  solution of \eqref{FGP} and the sequence $\{x^k\}$ be generated by Algorithm~\ref{GPG}. Then, we have 
    	\begin{equation}\label{eq:GPGMrate}
    	F(x^{k+1}_{ag})- F(x^*)
    	\leq \frac{2\gamma}{(2-\alpha)(k+1)^2}\left(\norm{\hat{x}^1-x^*}^2-\norm{\hat{x}^{k+1}-x^*}^2\right),
    	\end{equation}
    	where $\alpha=\frac{L}{\gamma}\in (0,1]$.
    \end{theorem}
    \begin{proof}
    	Form $\eqref{PGsxd}$ and $\eqref{PGxag}$, $x^{k+1}_{ag}-x^k_{md}=\theta_k(x^{k+1}-\hat{x}^k)$. Using this observation and the smoothness and convexity of $f(\cdot)$ (i.e., \eqref{PG_L}), we have 	
    	\begin{small}
    			\begin{equation}\label{GPGeq0}
    		\begin{aligned}
    		f(x^{k+1}_{ag})- f(x)&\leq f(x^k_{md})+\inp{x^{k+1}_{ag}-x^k_{md},\nabla f(x^k_{md})}+\frac{L}{2}\norm{x^{k+1}_{ag}-x^k_{md}}^2-f(x)\\
    		&=f(x^k_{md})+\inp{x^{k+1}_{ag}-x^k_{md},\nabla f(x^k_{md})}+\frac{L\theta_k^2}{2}\norm{x^{k+1}-\hat{x}^k}^2-f(x)\\
    		&= (1-\theta_k)(f(x^k_{md})+\inp{x^{k}_{ag}-x^k_{md},\nabla f(x^k_{md})}-f(x))\\
    		&\quad+\theta_k (f(x^k_{md})-f(x)+\inp{x^{k+1}-x^k_{md},\nabla f(x^k_{md})})+\frac{L\theta_k^2}{2}\norm{\hat{x}^k-x^{k+1}}^2\\
    		&\leq (1-\theta_k)(f(x^k_{ag})-f(x))+\theta_k (\inp{\nabla f(x^k_{md}),x^{k+1}-x})+\frac{L\theta_k^2}{2}\norm{\hat{x}^k-x^{k+1}}^2,
    		\end{aligned}
    		\end{equation}
    	\end{small}where the second equality follows from \eqref{PGxag} and the last inequality comes from the convexity of $f(\cdot)$. 
    	Moreover, from the convexity of $g(x)$ and \eqref{PGxag}, we have 
    	\begin{small}
    			\begin{equation*}
    		g(x^{k+1}_{ag})-g(x)\leq (1-\theta_k)(g(x^k_{ag})-g(x))+\theta_k (g(x^{k+1})-g(x)).
    		\end{equation*}
    	\end{small}Combining the above two inequality, we obtain
    	\begin{small}
    			\begin{equation}\label{GPGeq1}
    		\begin{aligned}
    		&\quad\;f(x^{k+1}_{ag})+g(x^{k+1}_{ag})- (f(x)+g(x))\\
    		&\leq (1-\theta_k)(f(x^k_{ag})+g(x^k_{ag})-f(x)-g(x))\\
    		&\quad+\theta_k (g(x^{k+1})-g(x)+\inp{\nabla f(x^k_{md}),x^{k+1}-x})+\frac{L\theta_k^2}{2}\norm{\hat{x}^k-x^{k+1}}^2.
    		\end{aligned}
    		\end{equation}
    	\end{small}By the first order optimality condition of \eqref{PGxk}, we have for any $x\in X$
    	\begin{equation*}
    	\begin{aligned}
    	&\quad \;g(x^{k+1})-g(x)+\inp{\nabla f(x^k_{md}),x^{k+1}-x}\\
    	&\leq \tau_k\inp{x^{k+1}-\hat{x}^k,x-x^{k+1}}\\
    	&=\frac{\tau_k}{2}\left(\norm{\hat{x}^k-x}^2-\norm{x^{k+1}-x}^2-\norm{\hat{x}^k-x^{k+1}}^2\right),
    	\end{aligned}
    	\end{equation*}
    	where the equality is from Lemma \ref{fact} (iii). 
    	Using the above inequality into \eqref{GPGeq1} and taking $x=x^*$, we have 
    	\begin{small}
    			\begin{equation*}
    		\begin{aligned}
    		&\quad\;f(x^{k+1}_{ag})+g(x^{k+1}_{ag})- f(x^*)-g(x^*)\\
    		&\leq (1-\theta_k)(f(x^k_{ag})-f(x^*)+g(x^k_{ag})-g(x^*))\\
    		&\quad+\frac{\tau_k\theta_k}{2}\norm{\hat{x}^k-x^*}^2
    		-\frac{\tau_k\theta_k}{2}\norm{x^{k+1}-x^*}^2-\left(\frac{\tau_k\theta_k}{2}-\frac{L\theta_k^2}{2}\right)\norm{\hat{x}^k-x^{k+1}}^2. 
    		\end{aligned}
    		\end{equation*}
    	\end{small}Recalling that $t_k=\cfrac{1}{\theta_k}$ and $\gamma=\cfrac{\tau_k}{\theta_k}$ in \eqref{PGpara}, we divide both sides of the above inequality by $\theta_k^2$ and obtain
    	\begin{equation}\label{GPGeq2}
    	\begin{aligned}
    	&\quad\;t^2_k(f(x^{k+1}_{ag})+g(x^{k+1}_{ag})- f(x^*)-g(x^*))\\
    	&\leq t_k(t_k-1)(f(x^k_{ag})+g(x^k_{ag})-f(x^*)-g(x^*))\\
    	&\quad+\frac{\gamma}{2}\left(\norm{\hat{x}^k-x^*}^2-\norm{x^{k+1}-x^*}^2\bm{-\left(1-\alpha\right)\norm{\hat{x}^k-x^{k+1}}^2}\right),
    	\end{aligned}
    	\end{equation}
    	where $\alpha=\cfrac{L}{\gamma}$. Moreover, we can rewrite \eqref{GPGeq2} in the following form: 
    	\begin{small}
    			\begin{equation*}
    		\begin{aligned}
    		&\quad\; t^2_k (F(x^{k+1}_{ag})-F(x^*))- t^2_{k-1}(F(x^{k}_{ag})-F(x^*))\\
    		&= t^2_k (F(x^{k+1}_{ag})-F(x^*))- t_k(t_k-1) (F(x^{k}_{ag})-F(x^*))\\
    		&\leq \frac{\gamma}{2(2-\alpha)}\left((2-\alpha)(\norm{\hat{x}^k-x^*}^2-\norm{x^{k+1}-x^*}^2-\left(1-\alpha\right)\norm{\hat{x}^k-x^{k+1}}^2)\right)\\
    		&=\frac{\gamma}{2(2-\alpha)}\left(\norm{\hat{x}^k-x^*}^2-\norm{x^{k+1}-x^*-(1-\alpha)(\hat{x}^k-x^{k+1})}^2\right)\\
    		&=\frac{\gamma}{2(2-\alpha)}\left(\norm{\hat{x}^k-x^*}^2-\norm{(\alpha-1)\hat{x}^k+(2-\alpha)x^{k+1}-x^*}^2\right)\\
    		&=\frac{\gamma}{2(2-\alpha)}\left(\norm{\hat{x}^k-x^*}^2-\norm{\hat{x}^{k+1}-x^*}^2\right),
    		\end{aligned}
    		\end{equation*}
    	\end{small}where the first equality follows from Lemma \ref{factt}, the second euqality is from Lemma \ref{fact} (ii) and the last equality is from the iteration of \eqref{PGhatx}. Applying the above inequality inductively with $t_0=0$ and $\hat{x}^1=x^1$, we have 
    	\begin{small}
    		\begin{equation}\label{bound1}
    		\begin{aligned}
    		t^2_k (F(x^{k+1}_{ag})-F(x^*))\leq \frac{\gamma}{2(2-\alpha)}\left(\norm{\hat{x}^1-x^*}^2-\norm{\hat{x}^{k+1}-x^*}^2\right),
    		\end{aligned}
    		\end{equation}
    	\end{small}which implies, according to  Lemma \ref{factt}, that
    	\begin{small}
    		\begin{equation*}
    	F(x^{k+1}_{ag})- F(x^*)
    	\leq \frac{2\gamma}{(2-\alpha)(k+1)^2}\left(\norm{\hat{x}^1-x^*}^2-\norm{\hat{x}^{k+1}-x^*}^2\right).
    	\end{equation*}
    	\end{small}This completes the proof.
    \end{proof}

\begin{remark}\label{rem:2}
	Observing the formulation \eqref{GPGeq2} in the proof of Theorem \ref{ThGPG}, we can see that when $\alpha=1$, the proof of GPGM is consistent with that of Nesterov's second APGM. The key difference between   Nesterov's second APGM and GPGM lies in the utilization of the term \bm{$-(1-\alpha)\| \hat{x}^k - x^{k+1}\|^2$} in  \eqref{GPGeq2} to design the extrapolation \eqref{PGhatx}, where we set $\gamma > L$. In contrast, Nesterov's second APGM directly uses $\tau_k = L\theta_k$, which shows $\alpha=1$, thereby ignoring the negative term $-\| \hat{x}^k - x^{k+1} \|^2$. This parameter setting allows GPGM to achieve the same convergence rate as Nesterov's second APGM but with a more general result corresponding to \eqref{eq:GPGMrate},
	as stated in Theorem \ref{ThGPG} (see Table \ref{Table 0}).
\end{remark}

\begin{remark}\label{rem:21}
	In the following, we test $\ell_1$ regularized logistic regression \eqref{lg} and provide some numerical evidence to show that the upper bound \eqref{bound1} from GPGM with $\alpha<1$
	\[
	t^2_k (F(x^{k+1}_{ag})-F(x^*))\leq \frac{\gamma}{2(2-\alpha)}\left(\|\hat{x}^1-x^*\|^2-\|\hat{x}^{k+1}-x^*\|^2\right)
	\]
	may be tighter than the upper bound from Nesterov's second APGM ($\alpha=1$)
	\[
	t^2_k (F(x^{k+1}_{ag})-F(x^*))\leq \frac{L}{2}\left(\|\hat{x}^1-x^*\|^2-\|x^{k+1}-x^*\|^2\right),
	\]
	that is 
	\[
	\frac{\gamma}{2(2-\alpha)}\left(\|\hat{x}^1-x^*\|^2-\|\hat{x}^{k+1}-x^*\|^2\right)\leq \frac{L}{2}\left(\|\hat{x}^1-x^*\|^2-\|x^{k+1}-x^*\|^2\right).
	\]
	
	We consider $\ell_1$ regularized logistic regression problem \eqref{lg} with different parameters $(m, n, s) = (3000, 300, 30)$, $(5000, 500, 50)$, and $(8000, 800, 80)$ and $\alpha=0.2, 0.4, 0.6, 0.8, 1$. The numerical results are presented in Figure \ref{zbFig1}  (a)-(c).
	We can see from Figure \ref{zbFig1} (a)-(c) that, in different $(m,n,s)$, the upper bounds from GPGM with $\alpha<1$ are tighter than those from Nesterov's second APGM ($\alpha$=1). This suggests that GPGM may lead to a larger decrease.
	\begin{figure}[htbp]
		\centering
		\subfigure[]{
			\includegraphics[scale=0.22]{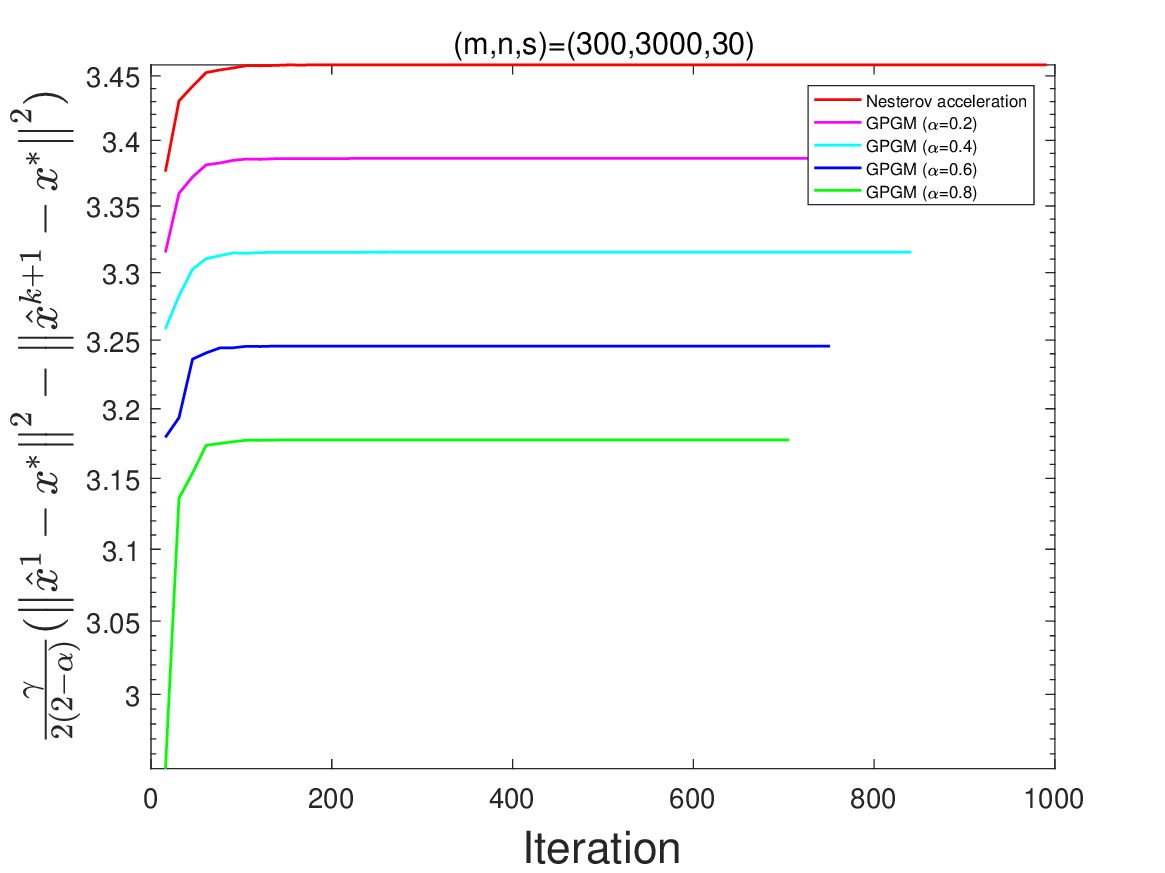}
		}
		\subfigure[]{
			\includegraphics[scale=0.22]{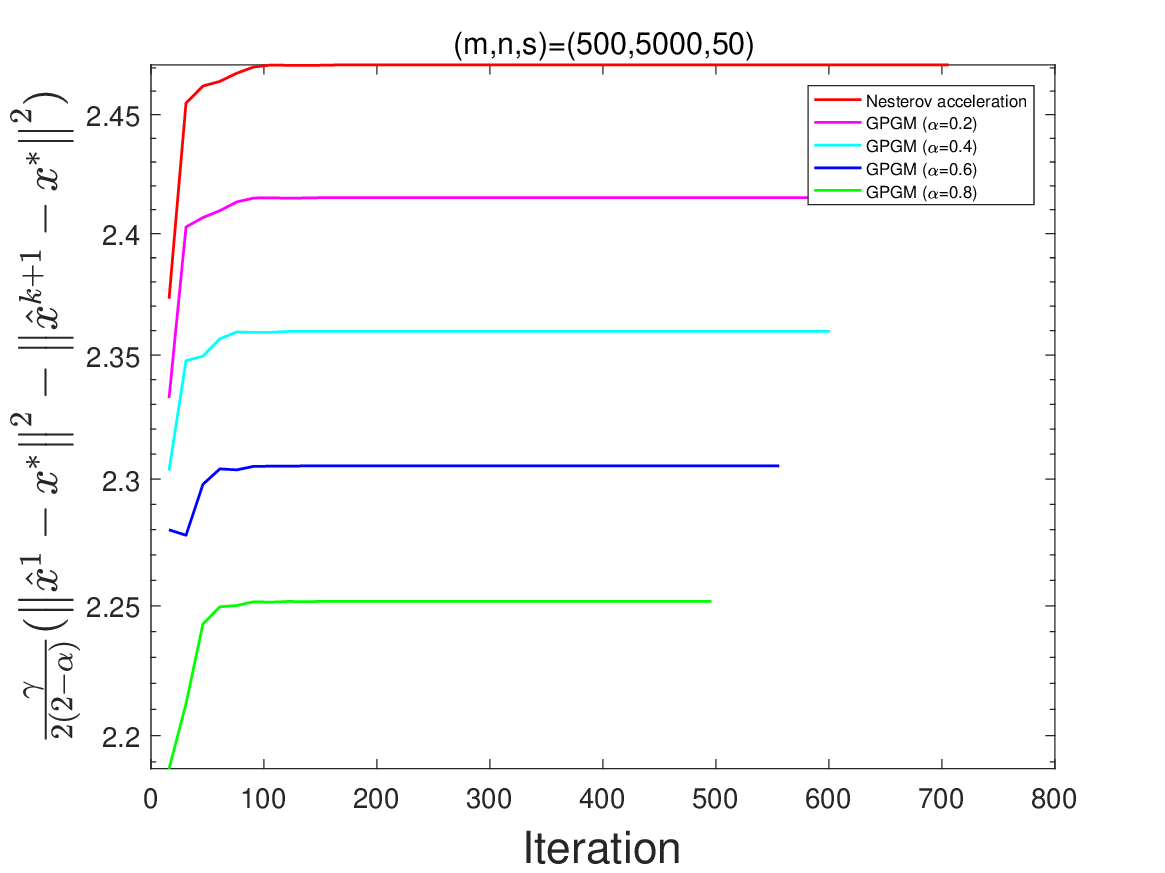}
		}
		\subfigure[]{
			\includegraphics[scale=0.22]{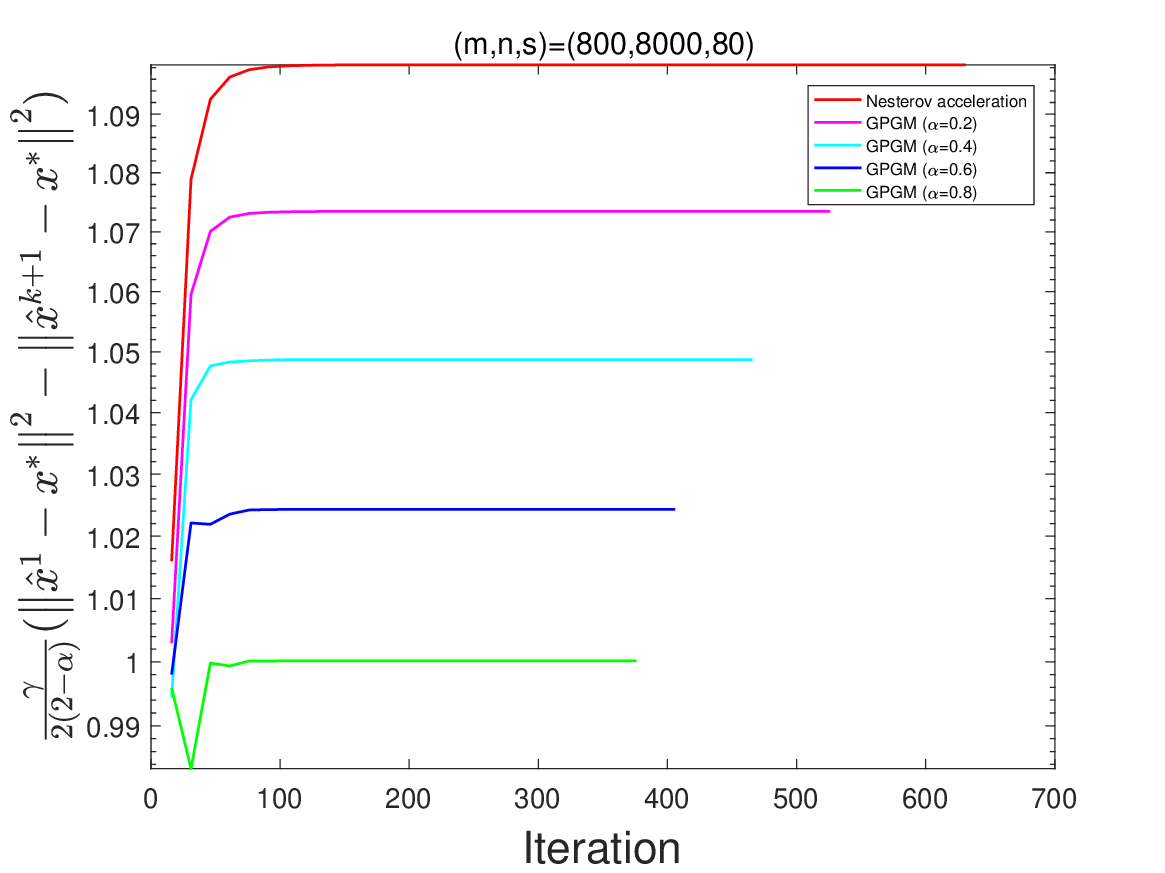}
		}
		\caption{The upper bounds of GPGM and Nesterov's second APGM}
		\label{zbFig1}
	\end{figure}
\end{remark}


\section{The G\"uler-type accelerated LALM}
From Section 2, we have learned that the G\"uler-type acceleration technique is applicable to gradient-based algorithms. In the case of linearly constrained composite convex programming problem \eqref{LCP}, we apply G\"uler-type acceleration to LALM. The proposed GLALM method is presented in Algorithm \ref{GALALM}.

\begin{algorithm}[h]  
	\caption{G\"uler-type accelerated linearized augmented Lagrangian method (GLALM)}  
	\label{GALALM}  
	\begin{algorithmic}[1]  
		\Require  initial point $x^1_{ag}=x^1=\hat{x}^1$ and set $z^1=\hat{z}^1=0$, $\alpha=\frac{2L}{\eta}$ 
		\State $k \leftarrow 1$  
		\Repeat  
		\State Generate
		\begin{align}
		&x^k_{md}=\left(1-\theta_k\right) x^k_{ag}+\theta_k \hat{x}^k\label{GALxmd}\\
		&x^{k+1}= \underset{x\in X}{\arg \min }\left\langle\nabla f(x^k_{md})-A^{\top} \hat{z}^k, x\right\rangle+g(x)+\frac{\kappa_k}{2}\|A x-b\|^2+\frac{1}{2}\left\|x-\hat{x}^k\right\|_{P^k}^2\label{GALx}\\
		&\bm{\hat{x}^{k+1}=(\alpha-1)\hat{x}^{k}+(2-\alpha)x^{k+1}} \label{GALhatx}\\
		&x^{k+1}_{ag}=\left(1-\theta_k\right) x^k_{ag}+\theta_k x^{k+1}\label{GALxag}\\
		&\bm{z^{k+1}=\hat{z}^k-\gamma_k(A x^{k+1}-b)}\label{GALz}\\
		&\hat{z}^{k+1}=(1-\kappa)\hat{z}^{k}+\kappa z^{k+1}\label{GALhatz}
		\end{align}
		\State $k \leftarrow k+1$
		\Until $\left(x^{k+1}_{ag}, z^{k+1}\right)$ satisfies the stop criterion. 
		\Ensure  $\left(x^{k+1}_{ag}, z^{k+1}\right)$  
	\end{algorithmic}  
\end{algorithm}
Note that if, in Algorithm \ref{GALALM}, $\alpha=1$ and $\kappa=1$, then $\hat{x}^k=x^k$ and $\hat{z}^k=z^k$. The Algorithm~\ref{GALALM}  becomes the  ALALM in \cite[Algorithm 1]{xu2017accelerated}. However, when $\alpha\in (0,1)$ and $\kappa\in (1,2)$, our GALM algorithm can accelerate both the primal variable $x^k$ and the dual variable $z^k$. In contrast, existing accelerated ALM algorithms, as described in \cite{xu2017accelerated,ke2017accelerated,he2010acceleration}, 
can only accelerate either the primal or dual variable. 

To analyze the convergence rate of Algorithm \ref{GALALM}, we first need the following lemma to establish the relationship between two iterations of Algorithm \ref{GALALM}.
\begin{lemma}\label{l:GLALM}
	Let $\left\{\left(x^k,\hat{x}^k, x^k_{ag}, \hat{z}^k,z^k\right)\right\}$ be the sequence generated by Algorithm \ref{GALALM}. Then for any $(x, z)$ such that $A x=b$, we have
	\begin{small}
			\begin{equation}\label{GALiteration}
		\begin{aligned}
		&\quad\; \left(F\left(x^{k+1}_{ag}\right)-F(x)-\inp{z, A x^{k+1}_{ag}-b}\right)-(1-\theta_k)\left(F\left(x^k_{ag}\right)-F(x)-\inp{z, A x^k_{ag}-b}\right) \\
		&\leq\frac{\theta_k}{2}\left(\norm{\hat{x}^k-x}^2_{P^k}-\norm{x^{k+1}-x}^2_{P^k}-\norm{\hat{x}^k-x^{k+1}}^2_{P^k}\right)+\frac{L\theta_k^2}{2}\left\|\hat{x}^k-x^{k+1}\right\|^2\\
		&\quad +\frac{\theta_k}{2\gamma_{k}}\left(\norm{\hat{z}^k-z}^2-\norm{z^{k+1}-z}^2+\norm{\hat{z}^k-z^{k+1}}^2\right)
		-\frac{\theta_k\kappa_{k}}{\gamma^2_{k}}\norm{\hat{z}^k-z^{k+1}}^2. 
		\end{aligned}
		\end{equation}
	\end{small}
\end{lemma}
\begin{proof}
	Using \eqref{LCP}, \eqref{GALxmd} and \eqref{GALxag}, we have the same inequality as \eqref{GPGeq0} in the following
	\begin{small}
		\begin{equation}\label{GALeq1}
		\begin{aligned}
		f(x^{k+1}_{ag})- f(x)
		\leq (1-\theta_k)(f(x^k_{ag})-f(x))+\theta_k (\inp{\nabla f(x^k_{md}),x^{k+1}-x})+\frac{L\theta_k^2}{2}\norm{\hat{x}^k-x^{k+1}}^2.
		\end{aligned}
		\end{equation}
	\end{small}By some simple calculations and using the relation of \eqref{GALxag}, we obtain
	\begin{small}
		\begin{equation}\label{GALeq2}
		\begin{aligned}
		&\quad\;\left(F(x^{k+1}_{ag})-F(x)-\inp{z, A x^{k+1}_{ag}-b}\right)-\left(1-\theta_k\right)\left(F(x^k_{ag})-F(x)-\inp{z, A x^k_{ag}-b}\right)\\
		&=\left[f(x^{k+1}_{ag})-\left(1-\theta_k\right) f(x^k_{ag})-\theta_k f(x)\right]+\left[g(x^{k+1}_{ag})-\left(1-\theta_k\right) g(x^k_{ag})-\theta_k g(x)\right] \\
		& \quad-\theta_k\inp{z, A x^{k+1}-b}\\
		&\leq \theta_k\inp{x^{k+1}-x,\nabla f(x^k_{md})}+\frac{L\theta_k^2}{2}\norm{\hat{x}^k-x^{k+1}}^2+\theta_k(g(x^{k+1})-g(x))-\theta_k\inp{z, A x^{k+1}-b}, 
		\end{aligned}
		\end{equation}
	\end{small}where the last inequality comes from \eqref{GALeq1} and the convexity of $g(x)$. 
	By the first order optimality condition of \eqref{GALx}, $Ax=b$ and \eqref{GALz}, we have
	\begin{small}
		\begin{equation*}
		\begin{aligned}
		0&\geq \inp{x^{k+1}-x, \nabla f(x^k_{md})-A^{\top} \hat{z}^k+\kappa_k A^{\top}(A x^{k+1}-b)+P^k\left(x^{k+1}-\hat{x}^k\right)}+g\left(x^{k+1}\right)-g(x)\\
		&=\inp{x^{k+1}-x, \nabla f(x^k_{md})-A^{\top} \hat{z}^k+\frac{\kappa_k}{\gamma_k} A^{\top}\left(\hat{z}^k-z^{k+1}\right)+P^k\left(x^{k+1}-\hat{x}^k\right)}+g(x^{k+1})-g(x) \\
		&= \inp{ x^{k+1}-x, \nabla f(x^k_{md})}+g(x^{k+1})-g(x)+\inp{x^{k+1}-x, P^k\left(x^{k+1}-\hat{x}^k\right)} \\
		& \quad+\inp{A(x^{k+1}-x),-\hat{z}^k+\frac{\kappa_k}{\gamma_k}\left(\hat{z}^k-z^{k+1}\right)}\\
		&= \left\langle x^{k+1}-x, \nabla f(x^k_{md})\right\rangle+g\left(x^{k+1}\right)-g(x)+\left\langle x^{k+1}-x, P^k\left(x^{k+1}-\hat{x}^k\right)\right\rangle \\
		&\quad +\left\langle A x^{k+1}-b, z-\hat{z}^k+\frac{\kappa_k}{\gamma_k}\left(\hat{z}^k-z^{k+1}\right)\right\rangle-\left\langle z, A x^{k+1}-b\right\rangle \\
		&=\left\langle x^{k+1}-x, \nabla f(x^k_{md})\right\rangle+g\left(x^{k+1}\right)-g(x)-\left\langle z, A x^{k+1}-b\right\rangle+\left\langle x^{k+1}-x, P^k\left(x^{k+1}-\hat{x}^k\right)\right\rangle \\
		&\quad +\left\langle\frac{1}{\gamma_k}\left(\hat{z}^k-z^{k+1}\right), z-\hat{z}^k+\frac{\kappa_k}{\gamma_k}\left(\hat{z}^k-z^{k+1}\right)\right\rangle\\
		&=\left\langle x^{k+1}-x, \nabla f(x^k_{md})\right\rangle+g\left(x^{k+1}\right)-g(x)-\left\langle z, A x^{k+1}-b\right\rangle+\left\langle x^{k+1}-x, P^k\left(x^{k+1}-\hat{x}^k\right)\right\rangle \\
		&\quad +\frac{1}{\gamma_k}\left\langle \hat{z}^k-z^{k+1}, z-\hat{z}^k\right\rangle
		+\frac{\kappa_{k}}{\gamma^2_{k}}\norm{\hat{z}^k-z^{k+1}}^2. 
		\end{aligned}
		\end{equation*}
	\end{small}Combining the above inequality into \eqref{GALeq2}, we obtain
	\begin{small}
		\begin{equation*}
		\begin{aligned}
		&\quad\; \left(F\left(x^{k+1}_{ag}\right)-F(x)-\inp{z, A x^{k+1}_{ag}-b}\right)-(1-\theta_k)\left(F\left(x^k_{ag}\right)-F(x)-\inp{z, A x^k_{ag}-b}\right) \\
		&  \leq \frac{\theta_k^2 L}{2}\left\|x^{k+1}-\hat{x}^k\right\|^2-\theta_k\left\langle x^{k+1}-x, P^k\left(x^{k+1}-\hat{x}^k\right)\right\rangle \\
		& \quad-\frac{\theta_k}{\gamma_k}\left\langle \hat{z}^k-z^{k+1}, z-\hat{z}^k\right\rangle
		-\frac{\theta_k\kappa_{k}}{\gamma^2_{k}}\norm{\hat{z}^k-z^{k+1}}^2\\
		&=\frac{\theta_k}{2}\left(\norm{\hat{x}^k-x}^2_{P^k}-\norm{x^{k+1}-x}^2_{P^k}-\norm{\hat{x}^k-x^{k+1}}^2_{P^k}\right)+\frac{L\theta_k^2}{2}\left\|\hat{x}^k-x^{k+1}\right\|^2\\
		&\quad +\frac{\theta_k}{2\gamma_{k}}\left(\norm{\hat{z}^k-z}^2-\norm{z^{k+1}-z}^2+\norm{\hat{z}^k-z^{k+1}}^2\right)
		-\frac{\theta_k\kappa_{k}}{\gamma^2_{k}}\norm{\hat{z}^k-z^{k+1}}^2, 
		\end{aligned}
		\end{equation*}
	\end{small}where the equality follows from Lemma \ref{fact}(iv).
\end{proof}

Now we analysis the convergence rate of GLALM. 
\begin{theorem}\label{ThGAL}
	Let $\left(x^*, z^*\right)$ be a KKT point of \eqref{LCP} and the sequence $\left\{\left(x^k,\hat{x}^k, x^k_{ag}, \hat{z}^k,z^k\right)\right\}$ be generated by Algorithm \ref{GALALM} with parameters set to 
\begin{small}
		\begin{equation}\label{GALpara}
	\theta_k=\frac{2}{k+1},\quad \gamma_k=k\gamma,\quad \kappa_k=\frac{\kappa \gamma k}{2},\quad P^k=\frac{\eta}{k}I,
	\end{equation}
\end{small}
	where $\eta>0$. Then, 
	\begin{small}
		\begin{equation}\label{GALeq4}
		\begin{aligned}
		&\left|F(x^{k+1}_{ag})-F(x^*)\right| \leq \frac{1}{k(k+1)}\left(\frac{\eta}{2-\alpha}\left\|x^1-x^*\right\|^2+\frac{\max \left\{\left(1+\left\|z^*\right\|\right)^2, 4\left\|z^*\right\|^2\right\}}{\gamma\kappa}\right)\\
		&\left\|A x^{k+1}_{ag}-b\right\|\leq \frac{1}{k(k+1)}\left(\frac{\eta}{2-\alpha}\left\|x^1-x^*\right\|^2+\frac{\max \left\{\left(1+\left\|z^*\right\|\right)^2, 4\left\|z^*\right\|^2\right\}}{\gamma\kappa}\right). 
		\end{aligned}
		\end{equation}
	\end{small}
\end{theorem}
\begin{proof}
	With the parameters given in \eqref{GALpara}, we multiply $k(k + 1)$ to both sides of \eqref{GALiteration} to have
	\begin{small}
		\begin{equation}\label{GALeq3}
		\begin{aligned}
		&\quad\;k(k+1)\left(F(x^{k+1}_{ag})-F(x)-\inp{z, A x^{k+1}_{ag}-b}\right)-(k-1)k\left(F(x^k_{ag})-F(x)-\inp{z, A x^k_{ag}-b}\right)\\
		&\leq \eta\left(\norm{\hat{x}^k-x}^2-\norm{x^{k+1}-x}^2\right)-\left(\eta-\frac{2kL}{k+1}\right)\norm{\hat{x}^k-x^{k+1}}^2\\
		&\quad +\frac{1}{\gamma}\left(\norm{\hat{z}^k-z}^2-\norm{z^{k+1}-z}^2-(\kappa-1)\norm{\hat{z}^k-z^{k+1}}^2\right)\\
		&\leq \eta\left(\norm{\hat{x}^k-x}^2-\norm{x^{k+1}-x}^2\bm{-\left(1-\alpha\right)\norm{\hat{x}^k-x^{k+1}}^2}\right)\\
		&\quad+\frac{1}{\gamma}\left(\norm{\hat{z}^k-z}^2-\norm{z^{k+1}-z}^2\bm{-(\kappa-1)\norm{\hat{z}^k-z^{k+1}}^2}\right),
		\end{aligned}
		\end{equation}
	\end{small}where the first inequality comes from $k/(k+1)\leq 1$ and $\alpha=\frac{2L}{\eta}$.
	Note that from Lemma \ref{fact}(i) and \eqref{GALhatx}, we have
	\begin{small}
			\begin{equation*}
		\begin{aligned}
		&\quad\;\eta\left(\norm{\hat{x}^k-x}^2-\norm{x^{k+1}-x}^2-\left(1-\beta\right)\norm{\hat{x}^k-x^{k+1}}^2\right)\\
		&= \frac{\eta}{2-\alpha}\left((2-\alpha)\left(\norm{\hat{x}^k-x}^2-\norm{x^{k+1}-x}^2-\left(1-\alpha\right)\norm{\hat{x}^k-x^{k+1}}^2\right)\right)\\
		&= \frac{\eta}{2-\alpha}\left(\norm{\hat{x}^k-x}^2-\norm{x^{k+1}-x-(1-\alpha)(\hat{x}^k-x^{k+1})}^2\right)\\
		&= \frac{\eta}{2-\alpha}\left(\norm{\hat{x}^k-x}^2-\norm{\hat{x}^{k+1}-x}^2\right). 
		\end{aligned}
		\end{equation*}
	\end{small}Analogously, from Lemma \ref{fact}(ii) and \eqref{GALhatz}, we have 
	\begin{small}
		\begin{equation*}
		\begin{aligned}
		&\quad\;\frac{1}{\gamma}\left(\norm{\hat{z}^k-z}^2-\norm{z^{k+1}-z}^2-(\kappa-1)\norm{\hat{z}^k-z^{k+1}}^2\right)\\
		&=\frac{1}{\kappa\gamma}\left(\kappa\left(\norm{\hat{z}^k-z}^2-\norm{z^{k+1}-z}^2-(\kappa-1)\norm{\hat{z}^k-z^{k+1}}^2\right)\right)\\
		&=\frac{1}{\kappa\gamma}\left(\norm{\hat{z}^k-z}^2-\norm{z^{k+1}-z-(\kappa-1)(\hat{z}^k-z^{k+1})}^2\right)\\
		&=\frac{1}{\kappa\gamma}\left(\norm{\hat{z}^k-z}^2-\norm{\hat{z}^{k+1}-z}^2\right).
		\end{aligned}
		\end{equation*}
	\end{small}Combining the above equality into \eqref{GALeq3} and taking $x=x^*$, we have 
	\begin{small}
			\begin{equation}\label{eq:bound2}
		\begin{aligned}
		&\quad\;k(k+1)\left(F(x^{k+1}_{ag})-F(x^*)-\inp{z, A x^{k+1}_{ag}-b}\right)-(k-1)k\left(F(x^k_{ag})-F(x^*)-\inp{z, A x^k_{ag}-b}\right)\\
		&\leq\frac{\eta}{2-\alpha}\left(\norm{\hat{x}^k-x^*}^2-\norm{\hat{x}^{k+1}-x^*}\right)
		+\frac{1}{\kappa\gamma}\left(\norm{\hat{z}^k-z}^2-\norm{\hat{z}^{k+1}-z}^2\right).
		\end{aligned}
		\end{equation}
	\end{small}Applying the above inequality inductively, we have
	\begin{small}
		\begin{equation*}
		\begin{aligned}
		k(k+1)\left(F(x^{k+1}_{ag})-F(x^*)-\inp{z, A x^{k+1}_{ag}-b}\right)
		\leq \frac{\eta}{2-\alpha}\left(\norm{\hat{x}^1-x^*}^2\right)
		+\frac{1}{\kappa\gamma}\left(\norm{\hat{z}^1-z}^2\right).
		\end{aligned}
		\end{equation*}
	\end{small}Noting $\hat{z}^1=0$, $\hat{x}^1=x^1$ and applying Lemmas \ref{rate1} and \ref{rate2} with $\rho=\max \left\{1+\left\|z^*\right\|, 2\left\|z^*\right\|\right\}$, we obtain
	\begin{small}
			\begin{equation*}
		\begin{aligned}
		&\left|F(x^{k+1}_{ag})-F(x^*)\right| \leq \frac{1}{k(k+1)}\left(\frac{\eta}{2-\alpha}\left\|x^1-x^*\right\|^2+\frac{\max \left\{\left(1+\left\|z^*\right\|\right)^2, 4\left\|z^*\right\|^2\right\}}{\gamma\kappa}\right)\\
		&\left\|A x^{k+1}_{ag}-b\right\|\leq \frac{1}{k(k+1)}\left(\frac{\eta}{2-\alpha}\left\|x^1-x^*\right\|^2+\frac{\max \left\{\left(1+\left\|z^*\right\|\right)^2, 4\left\|z^*\right\|^2\right\}}{\gamma\kappa}\right). 
		\end{aligned}
		\end{equation*}
	\end{small}
	This completes the proof.
\end{proof}
\begin{remark}\label{rem:3}
	In a similar manner to GPGM, when $\alpha=1$ and $\kappa=1$, GLALM simplifies to ALALM as described in \cite{xu2017accelerated}. The key modification in GLALM occurs when $\alpha\in(0,1)$ and $\kappa \in (1,2)$. In this case,  the terms \bm{$-\left(1-\alpha\right)\| \hat{x}^k-x^{k+1} \|^2$} and \bm{$-(\kappa-1)\| \hat{z}^k-z^{k+1} \|^2$} in formulation \eqref{GALeq3} from the proof of Theorem \ref{ThGAL} are negative. These terms are employed to construct a new extrapolations \eqref{GALhatx} and \eqref{GALhatz}. The proposed GLALM matches the optimal convergence rate of $O(\frac{1}{k^2})$ with 
	more general coefficients selection compared to the existing accelerated LALM presented in \cite{xu2017accelerated},  as shown in Table \ref{Table 0}.
\end{remark}

\begin{remark}
	Similar to Remark \ref{rem:21}, we provide some numerical evidence to show that the upper bound in \eqref{eq:bound2} from GLALM with $\alpha<1$
	\[
	\frac{\eta}{2-\alpha}\left(\|\hat{x}^1-x^*\|^2- \|\hat{x}^{k+1}-x^*\|^2\right)
	+\frac{1}{\kappa\gamma}\left(\|\hat{z}^1-z\|^2- \|\hat{z}^{k+1}-z\|^2\right)
	\]
	may be tighter than the upper bound with $\alpha=1$ as follows
	\[
	2L\left(\|\hat{x}^1-x^*\|^2- \|\hat{x}^{k+1}-x^*\|^2\right)
	+\frac{1}{\kappa\gamma}\left(\|\hat{z}^1-z\|^2- \|\hat{z}^{k+1}-z\|^2\right).
	\]
	We consider nonnegative linearly constrained quadratic programming (33) the following parameter settings: $(m, n) = (80,1000)$, $(100, 1000)$, and $(120,1500)$, $\alpha=0.2, 0.4, 0.8, 0.8, 1$ and $\kappa=1, 1.5$. The numerical results are presented in Figure \ref{zbFig2}. 
	We can see from Figure \ref{zbFig2} that, in different $(m,n)$ and $\kappa$, the upper bounds from GLALM with $\alpha<1$ are tighter than those from the case $\alpha$=1 (the ALALM is the case when $\alpha=1$ and $\kappa=1$). This suggests that GLALM may lead to a larger decrease. 
		
	\begin{figure}[htbp]
		\centering
		\subfigure[]{
			\includegraphics[scale=0.22]{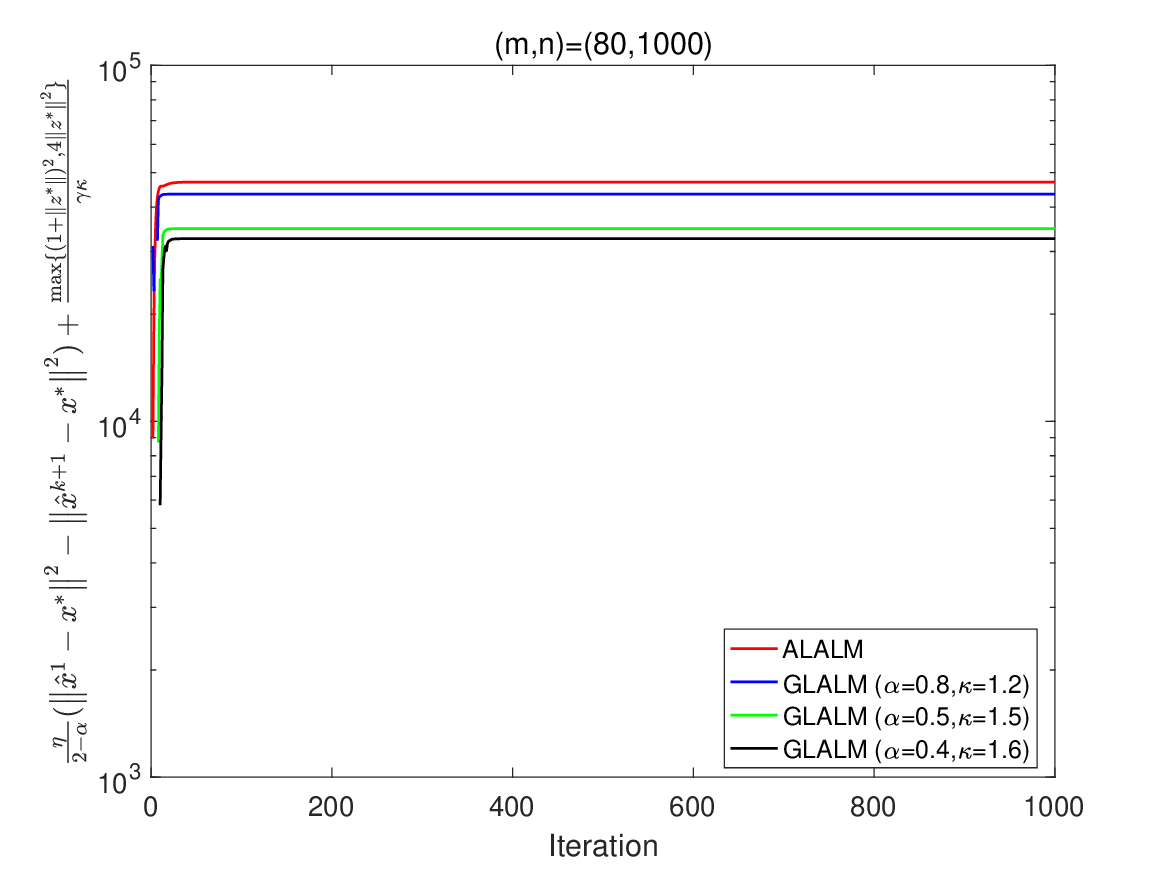}
		}
		\subfigure[]{
			\includegraphics[scale=0.22]{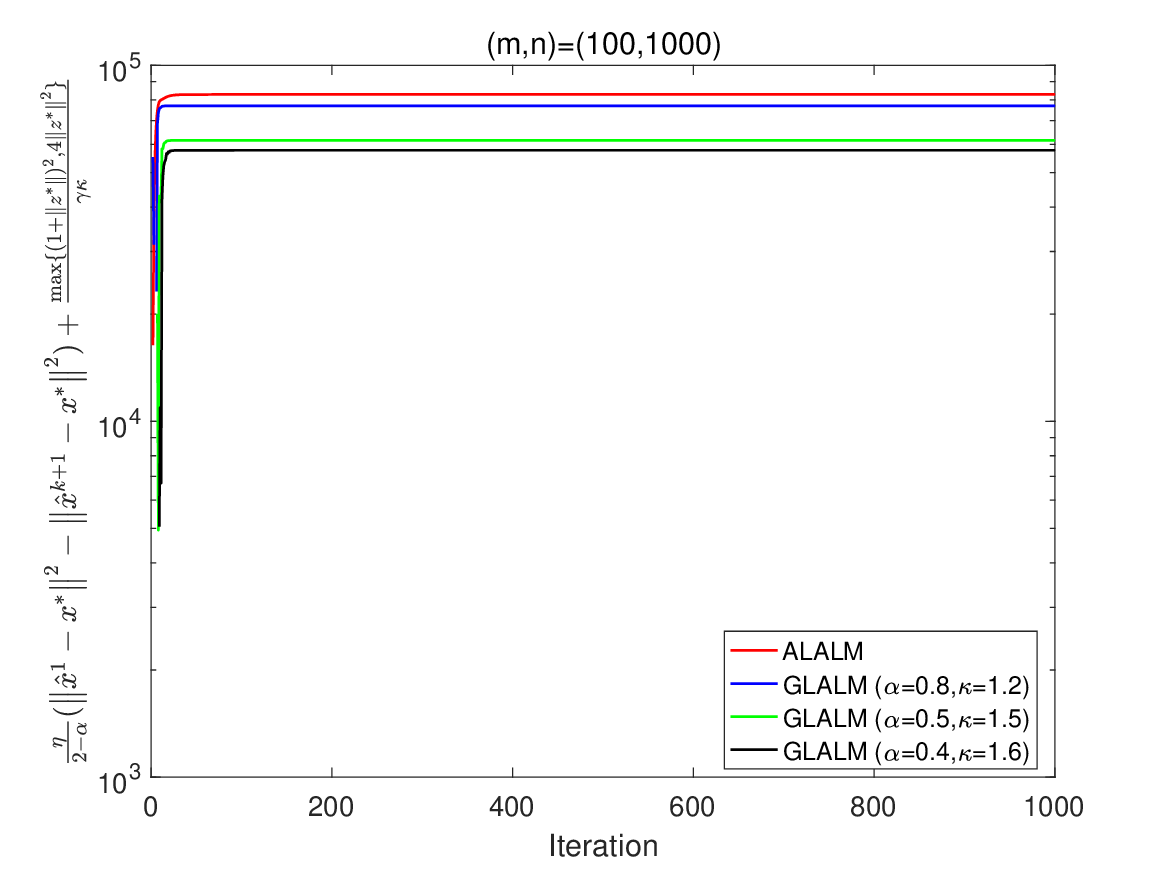}
		}
		\subfigure[]{
			\includegraphics[scale=0.22]{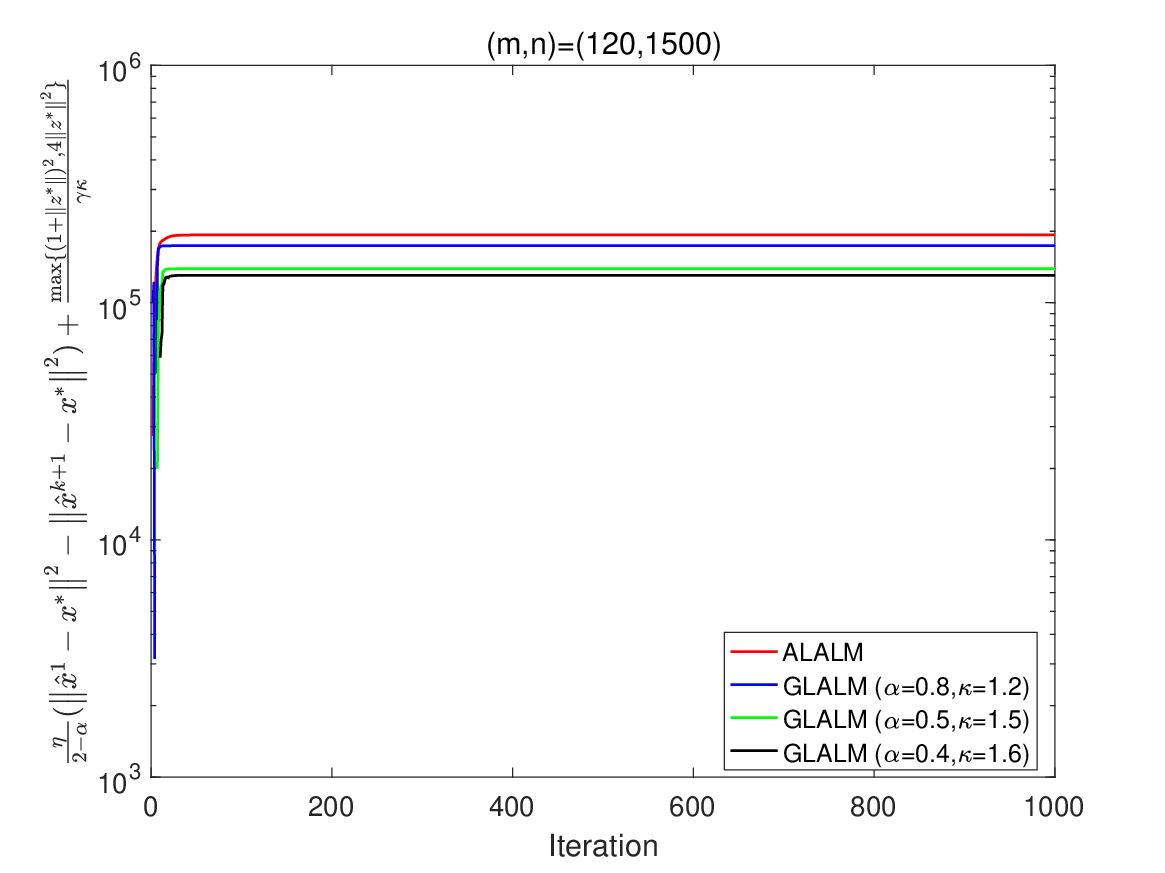}
		}
		\caption{The upper bounds of GLALM and ALALM}
		\label{zbFig2}
	\end{figure}
\end{remark}

\section{The G\"uler-type accelerated L-ADMM}

After accumulating the experience of improving PGM and LALM in Sections 2 and 3, we now discuss in this section how to incorporate the G\"uler acceleration technique into L-ADMM. This leads to the proposed Algorithm \ref{GADMM} named GLADMM.

\begin{algorithm}[h]  
	\caption{G\"uler-type accelerated L-ADMM (GLADMM)}  
	\label{GADMM}  
	\begin{algorithmic}[H]  
		\Require  Choose $x^1 \in X$ and $y^1 \in Y$ such that $B y^1=A x^1+b$. Set $x^1_{ag}=\hat{x}^1=x^1, y^1_{ag}=\hat{y}^1=y^1$ and $z^1_{ag}=\hat{z}^1=z^1=0$. 
		\State $k \leftarrow 1$  
		\Repeat  
		\State Generate
		\begin{align}
		&x^k_{md}=\left(1-\theta_k\right) x^k_{ag}+\theta_k \hat{x}^k\label{GADMMxmd}\\
		&\begin{aligned}
		x^{k+1}= & \underset{x \in X}{\operatorname{argmin}}\left\langle\nabla f(x^k_{md}),x\right\rangle 
		+\frac{\lambda_k}{2}\left\|B \hat{y}^k-A x-b\right\|^2\\
		&~~~~~~~~~~~~+\left\langle \hat{z}^k, A x\right\rangle+\frac{\eta_k}{2}\left\|x-\hat{x}^k\right\|^2
		\end{aligned}\label{GADMMx}\\
		&\bm{\hat{x}^{k+1}=(2-\alpha)x^{k+1}+(\alpha-1)\hat{x}^k}\label{GADMMhatx}\\
		&x^{k+1}_{ag}=\left(1-\theta_k\right) x^k_{a g}+\theta_k x^{k+1}\label{GADMMxag}\\
		&y^{k+1}=\underset{y \in Y}{\operatorname{argmin}}~ g(y)-\left\langle \hat{z}^k B y\right\rangle+\frac{\tau_k}{2}\left\|B y-A x^{k+1}-b\right\|^2\label{GADMMy}\\
		&\bm{\hat{y}^{k+1}=(2-\beta)y^{k+1}+(\beta-1)\hat{y}^k}\label{GADMMhaty}\\
		&y^{k+1}_{ag}=\left(1-\theta_k\right) y^k_{ag}+\theta_k y^{k+1}\label{GADMMyag}\\
		&z^{k+1}=\hat{z}^k-\gamma_k(B y^{k+1}-A x^{k+1}-b)\label{GADMMz}\\
		&\bm{\hat{z}^{k+1}=(1-\kappa)\hat{z}^k+\kappa z^{k+1}}\label{GADMMhatz}\\
		&z^{k+1}_{ag}=\left(1-\theta_k\right) z^k_{ag}+\theta_k z^{k+1}\label{GADMMzag}
		\end{align}
		\State $k \leftarrow k+1$
		\Until $\left(x^{k+1}_{ag}, y^{k+1}_{ag}\right)$ satisfies the stop criterion.
		\Ensure  $\left(x^{k+1}_{ag}, y^{k+1}_{ag}\right)$ 
	\end{algorithmic}  
\end{algorithm}
Note that if, in Algorithm \ref{GADMM}, $\alpha=1$, $\beta=1$ and $\kappa=1$, then $\hat{x}^{k}=x^{k}$, $\hat{y}^{k}=y^{k}$ and $\hat{z}^{k}=z^{k}$, and the  Algorithm \ref{GADMM}  becomes accelerated AL-ADMM in \cite{ouyang2015accelerated}.

To study the convergence analysis of Algorithm \ref{GADMM}, we first define the following gap function.\\
{\bf Gap function:} For any $\tilde{w}=(\tilde{x}, \tilde{y}, \tilde{z})$ and $w=(x, y, z)$, we define
\begin{equation}\label{GADMMGap}
\begin{aligned}
Q(\tilde{x}, \tilde{y}, \tilde{z} ; x, y, z):= & {[f(x)+g(y)-\langle\tilde{z}, B y-A x-b\rangle] } \\
& -[f(\tilde{x})+g(\tilde{y})-\langle z, B \tilde{y}-A \tilde{x}-b\rangle] .
\end{aligned}
\end{equation}

\begin{lemma}\label{l:GLADMM}
	Let $(x^*,y^*)$ denote a solution of \eqref{AECCO} and  be the sequence $\left\{\left(x^k_{ag}, y^k_{ag}, z^k_{ag}\right)\right\}$ be generated by
	Algorithm \ref{GADMM}. There exists $\alpha\in (0,1)$, $\beta\in (0,1)$ and $\kappa>1$ such that $\frac{L\theta_k}{\eta_k}\leq \alpha$, $\frac{1}{\xi_k}\leq \beta$, $\frac{\tau_k+(1-\xi_k)\lambda_k}{\gamma_k}\geq \kappa $. Let $\theta_1=1$ and
	\[
	\Gamma_k= \begin{cases}\Gamma_1 & \text { when } k=1, \\ \left(1-\theta_{k}\right) \Gamma_{k-1} & \text { when } k>1 .\end{cases}
	\]
	Then, we have for any $z$
	\begin{footnotesize}
		\begin{equation}\label{GADMMres}
		\begin{aligned}
		&\quad\; \frac{1}{\Gamma_k} Q(x^*,y^*,z;w^{k+1}_{ag})-\frac{1-\theta_k}{\Gamma_k} Q(x^*,y^*,z; w^k_{a g}) \\
		&\leq   \frac{\theta_k}{\Gamma_k}\left\{\frac{\eta_k}{2(2-\alpha)}\left(\left\|\hat{x}^k-x^*\right\|^2-\norm{\hat{x}^{k+1}-x^*}^2\right)
		+\frac{\lambda_k}{2(2-\beta)}\left(\norm{B \hat{y}^k-B y^*}^2-\norm{B\hat{y}^{k+1}-By^*}^2\right)\right.\\
		&\quad
		\left.+\frac{1}{2\gamma_k\kappa}\left(\norm{\hat{z}^k-z}^2-\norm{\hat{z}^{k+1}-z}^2\right)
		+\frac{\tau_k-\lambda_k}{2}\left\|A\left(x^{k+1}-x^*\right)\right\|^2 
		-\frac{\tau_k-\lambda_k}{2}\left\|B y^{k+1}-By^*\right\|^2\right\}.
		\end{aligned}
		\end{equation}
	\end{footnotesize}
\end{lemma}
\begin{proof}
	Similar to the inequality \eqref{GPGeq0}, by \eqref{AECCO}, \eqref{GADMMxmd} and \eqref{GADMMxag}, we have
	\begin{small}
		\begin{equation}\label{GADMMeq1}
		\begin{aligned}
		f(x^{k+1}_{ag})
		\leq (1-\theta_k)f(x^k_{ag})+\theta_kf(x)+\theta_k (\inp{\nabla f(x^k_{md}),x^{k+1}-x})+\frac{L\theta_k^2}{2}\norm{\hat{x}^k-x^{k+1}}^2.
		\end{aligned}
		\end{equation}
	\end{small}By \eqref{GADMMGap}, \eqref{GADMMyag}, \eqref{GADMMzag}, \eqref{GADMMeq1} and the convexity of $g(\cdot)$, we conclude that
	\begin{small}
			\begin{equation}\label{GADMMeq2}
		\begin{aligned}
		&\quad\;Q(w;w^{k+1}_{ag})-\left(1-\theta_k\right) Q(w; w^k_{ag}) \\
		&=  {\left[f\left(x^{k+1}_{ag}\right)+g(y^{k+1}_{ag})
			-\left\langle z, B y^{k+1}_{ag}-A x^{k+1}_{ag}-b\right\rangle\right]
			-\left[f(x)+g(y)-\left\langle z^{k+1}_{ag}, B y-A x-b\right\rangle\right] } \\
		&\quad -\left(1-\theta_k\right)\left[f(x^k_{ag})+g(y^k_{ag})
		-\left\langle z, B y^k_{ag}-A x^k_{ag}-b\right\rangle\right] \\
		&\quad +\left(1-\theta_k\right)\left[f(x)+g(y)-\left\langle z^k_{ag}, B y-A x-b\right\rangle\right] \\
		&= {\left[f(x^{k+1}_{ag})-\left(1-\theta_k\right) f(x^k_{ag})-\theta_k f(x)\right]
			+\left[g(y^{k+1}_{ag})-\left(1-\theta_k\right) g(y^k_{ag})-\theta_k g(y)\right] } \\
		&\quad -\theta_k\left\langle z, B y^{k+1}-A x^{k+1}-b\right\rangle+\theta_k\left\langle z^{k+1}, B y-A x-b\right\rangle \\
		&\leq \theta_k\left\langle\nabla f\left(x^k_{md}\right), x^{k+1}-x\right\rangle+\theta_k\left(g(y^{k+1})-g(y)\right)
		+\frac{L \theta_k^2}{2}\left\|x^{k+1}-\hat{x}^k\right\|^2 \\
		& \quad-\theta_k\left\langle z, B y^{k+1}-A x^{k+1}-b\right\rangle+\theta_k\left\langle z^{k+1}, B y-A x-b\right\rangle. 
		\end{aligned}
		\end{equation}
	\end{small}Next, we examine the optimality conditions in \eqref{GADMMx} and \eqref{GADMMy}. For all $x \in X$, $y \in Y$ and $By-Ax=b$, we have
	\begin{small}
			\begin{equation*}
		\begin{aligned}
		&\left\langle\nabla f(x^k_{md})+\eta_k(x^{k+1}-\hat{x}^k), x^{k+1}-x\right\rangle-\left\langle\lambda_k\left(B \hat{y}^k-A x^{k+1}-b\right)-\hat{z}^k, A\left(x^{k+1}-x\right)\right\rangle \leq 0\\
		&g(y^{k+1})-g(y)+\left\langle\tau_k\left(B y^{k+1}-A x^{k+1}-b\right)-\hat{z}^k, B(y^{k+1}-y)\right\rangle\leq 0. 
		\end{aligned}
		\end{equation*}
	\end{small}Observing from \eqref{GADMMz} that $B y^{k+1}-A x^{k+1}-b=\left(\hat{z}^k-z^{k+1}\right) / \gamma_k$ and 
	$B \hat{y}^k-A x^{k+1}-b=\left(\hat{z}^k-z^{k+1}\right) / \gamma_k+B(\hat{y}^k-y^{k+1})$, the optimality conditions become
	\begin{small}
			\begin{equation*}
		\begin{aligned}
		&\begin{aligned}
		&\left\langle\nabla f(x^k_{md})+\eta_k(x^{k+1}-\hat{x}^k),x^{k+1}-x\right\rangle
		+\left\langle\left(\frac{\lambda_k}{\gamma_k}-1\right)\left(\hat{z}^k-z^{k+1}\right)-z^{k+1},-A\left(x^{k+1}-x\right)\right\rangle \\
		&+\lambda_k\left\langle B\left(\hat{y}^k-y^{k+1}\right),-A\left(x^{k+1}-x\right)\right\rangle \leq 0,
		\end{aligned}\\
		&g(y^{k+1})-g(y)+\left\langle\left(\frac{\tau_k}{\gamma_k}-1\right)\left(\hat{z}^k-z^{k+1}\right)-z^{k+1}, B\left(y^{k+1}-y\right)\right\rangle \leq 0.
		\end{aligned}
		\end{equation*}
	\end{small}Therefore,
	\begin{small}
			\begin{equation}\label{GADMMeq3}
		\begin{aligned}
		&\quad\left\langle\nabla f(x^k_{md}),x^{k+1}-x\right\rangle+g(y^{k+1})-g(y)
		-\left\langle z, B y^{k+1}-A x^{k+1}-b\right\rangle+\left\langle z^{k+1}, B y-A x-b\right\rangle\\
		&\leq  \left\langle\eta_k\left(\hat{x}^k-x^{k+1}\right), x^{k+1}-x\right\rangle+\left\langle z^{k+1}-z, B y^{k+1}-A x^{k+1}-b\right\rangle \\
		&\quad +\left\langle\left(\frac{\lambda_k}{\gamma_k}-1\right)\left(\hat{z}^k-z^{k+1}\right),A\left(x^{k+1}-x\right)\right\rangle
		-\left\langle\left(\frac{\tau_k}{\gamma_k}-1\right)\left(\hat{z}^k-z^{k+1}\right), B\left(y^{k+1}-y\right)\right\rangle \\
		&\quad+\lambda_k\left\langle B\left(\hat{y}^k-y^{k+1}\right),A\left(x^{k+1}-x\right)\right\rangle. 
		\end{aligned}
		\end{equation}
	\end{small}Three observations on the right-hand side of \eqref{GADMMeq3} are in place. First, by \eqref{GADMMz} we have
	\begin{small}
			\begin{equation}
		\begin{aligned}
		&\quad\inp{\eta_k\left(\hat{x}^k-x^{k+1}\right), x^{k+1}-x}+\left\langle z^{k+1}-z, B y^{k+1}-A x^{k+1}-b\right\rangle\\
		&=\eta_k\inp{\hat{x}^k-x^{k+1}, x^{k+1}-x}+\frac{1}{\gamma_k}\left\langle z^{k+1}-z, \hat{z}^k-z^{k+1}\right\rangle \\
		&=\frac{\eta_k}{2}\left(\left\|\hat{x}^k-x\right\|^2-\left\|x^{k+1}-x\right\|^2-\left\|\hat{x}^k-x^{k+1}\right\|^2\right) \\
		&\quad+\frac{1}{2 \gamma_k}\left(\left\|\hat{z}^k-z\right\|^2-\left\|z^{k+1}-z\right\|^2-\left\|\hat{z}^k-z^{k+1}\right\|^2\right),
		\end{aligned}
		\end{equation}
	\end{small}and second, by \eqref{GADMMz} we can see that
	\begin{small}
			\begin{equation*}
		\begin{aligned}
		B\left(y^{k+1}-y\right)&=\frac{1}{\gamma_k}\left(\hat{z}^k-z^{k+1}\right)+\left(A x^{k+1}-A x\right)-(B y-A x-b)\\
		&=\frac{1}{\gamma_k}\left(\hat{z}^k-z^{k+1}\right)+\left(A x^{k+1}-A x\right),
		\end{aligned}
		\end{equation*}
	\end{small}and 
	\begin{footnotesize}
			\begin{equation}
		\begin{aligned}
		&\quad\left\langle\left(\frac{\lambda_k}{\gamma_k}-1\right)\left(\hat{z}^k-z^{k+1}\right),A\left(x^{k+1}-x\right)\right\rangle \\
		& \quad-\left\langle\left(\frac{\tau_k}{\gamma_k}-1\right)\left(\hat{z}^k-z^{k+1}\right), \frac{1}{\gamma_k}\left(\hat{z}^k-z^{k+1}\right)+\left(A x^{k+1}-A x\right)\right\rangle\\
		& =\frac{\tau_k-\lambda_k}{\gamma_k}\left\langle \hat{z}^k-z^{k+1},-A\left(x^{k+1}-x\right)\right\rangle
		-\frac{\tau_k-\gamma_k}{\gamma_k^2}\left\|\hat{z}^k-z^{k+1}\right\|^2 \\
		& =\frac{\tau_k-\lambda_k}{2}\left[\frac{1}{\gamma_k^2}\norm{\hat{z}^k-z^{k+1}}^2
		+\left\|A\left(x^{k+1}-x\right)\right\|^2-\left\|\frac{1}{\gamma_k}\left(\hat{z}^k-z^{k+1}\right)+A\left(x^{k+1}-x\right)\right\|^2\right] \\
		&\quad -\frac{\tau_k-\gamma_k}{\gamma_k^2}\left\|\hat{z}^k-z^{k+1}\right\|^2 \\
		& =\frac{\tau_k-\lambda_k}{2}\left[\frac{1}{\gamma_k^2}\norm{\hat{z}^k-z^{k+1}}^2
		+\left\|A\left(x^{k+1}-x\right)\right\|^2
		-\left\|B y^{k+1}-A x-b\right\|^2\right]-\frac{\tau_k-\gamma_k}{\gamma_k^2}\left\|\hat{z}^k-z^{k+1}\right\|^2,
		\end{aligned}
		\end{equation}
	\end{footnotesize}where the second equality comes from $\inp{a,b}=1/2(a^2+b^2-(a-b)^2)$. 
	Third, from \eqref{GADMMeq3} we have
	\begin{small}
			\begin{equation}\label{GADMMeq5}
		\begin{aligned}
		& \quad\;\lambda_k\left\langle B\left(\hat{y}^k-y^{k+1}\right),A\left(x^{k+1}-x\right)\right\rangle \\
		&= \frac{\lambda_k}{2}\left(\left\|B \hat{y}^k-A x-b\right\|^2-\left\|B y^{k+1}-A x-b\right\|^2\right. \\
		& \left.\quad+\left\|B y^{k+1}-A x^{k+1}-b\right\|^2-\left\|B \hat{y}^k-A x^{k+1}-b\right\|^2\right) \\
		& =\frac{\lambda_k}{2}\left(\left\|B \hat{y}^k-A x-b\right\|^2-\left\|B y^{k+1}-A x-b\right\|^2\right) \\
		&\quad +\frac{\lambda_k}{2 \gamma_k^2}\left\|\hat{z}^k-z^{k+1}\right\|^2
		-\frac{\lambda_k}{2}\left\|B \hat{y}^k-B y^{k+1}+\frac{1}{\gamma_k}\left(\hat{z}^k-z^{k+1}\right)\right\|^2\\
		& \leq \frac{\lambda_k}{2}\left(\left\|B \hat{y}^k-A x-b\right\|^2-\left\|B y^{k+1}-A x-b\right\|^2\right)  \\
		&\quad +\frac{\lambda_k}{2 \gamma_k^2}\left\|\hat{z}^k-z^{k+1}\right\|^2
		+\frac{\lambda_k}{2}\left(\left(\frac{1}{\xi_k}-1\right)\norm{B \hat{y}^k-B y^{k+1}}^2+\frac{1}{\gamma^2_k}\left(\xi_k-1\right)\norm{\hat{z}^k-z^{k+1}}^2\right)\\
		& = \frac{\lambda_k}{2}\left(\left\|B \hat{y}^k-A x-b\right\|^2-\left\|B y^{k+1}-A x-b\right\|^2\right)  \\
		&\quad +\frac{\lambda_k\xi_k}{2 \gamma_k^2}\left\|\hat{z}^k-z^{k+1}\right\|^2
		+\frac{\lambda_k}{2}\left(\frac{1}{\xi_k}-1\right)\norm{B \hat{y}^k-B y^{k+1}}^2,
		\end{aligned}
		\end{equation}
	\end{small}where the last inequality comes from $-(a+b)^2\leq (1/\xi-1)a^2+(\xi-1)b^2$ for $\xi>0$. 
	Applying \eqref{GADMMeq3}-\eqref{GADMMeq5} to \eqref{GADMMeq2}, we have
\begin{small}
		\begin{equation*}
	\begin{aligned}
	&\quad\; \frac{1}{\Gamma_k} Q(w;w^{k+1}_{ag})-\frac{1-\theta_k}{\Gamma_k} Q(w; w^k_{a g}) \\
	& \leq \frac{\theta_k}{\Gamma_k}\left\{ \frac{\eta_k}{2}\left(\left\|\hat{x}^k-x\right\|^2-\left\|x^{k+1}-x\right\|^2
	-\left(1-\frac{L\theta_k}{\eta_k}\right)\left\|\hat{x}^k-x^{k+1}\right\|^2\right)\right.\\
	&\quad+\frac{1}{2 \gamma_k}\left(\left\|\hat{z}^k-z\right\|^2-\left\|z^{k+1}-z\right\|^2\right)
	-\left(\frac{\tau_k+(1-\xi_k)\lambda_k}{\gamma_k}-1\right)\left\|\hat{z}^k-z^{k+1}\right\|^2 \\
	&\quad+\frac{\lambda_k}{2}\left(\left\|B \hat{y}^k-A x-b\right\|^2-\left\|B y^{k+1}-A x-b\right\|^2
	-\left(1-\frac{1}{\xi_k}\right)\norm{B \hat{y}^k-B y^{k+1}}^2\right) \\
	&\quad \left.+\frac{\tau_k-\lambda_k}{2}\left\|A\left(x^{k+1}-x\right)\right\|^2 
	-\frac{\tau_k-\lambda_k}{2}\left\|B y^{k+1}-A x-b\right\|^2\right\}.
	\end{aligned}
	\end{equation*}
\end{small}Form the set of $\frac{L\theta_k}{\eta_k}\leq \alpha \in (0,1)$, $\frac{1}{\xi_k}\leq \beta\in(0,1)$, $\frac{\tau_k+(1-\xi_k)\lambda_k}{\gamma_k}\geq \kappa >1$, letting $x=x^*$ and $y=y^*$ and, using $By^*=Ax^*+b$ in the above inequality, we have  
	\begin{small}
			\begin{equation}\label{GADMMeq6}
		\begin{aligned}
		&\quad\; \frac{1}{\Gamma_k} Q(x^*,y^*,z;w^{k+1}_{ag})-\frac{1-\theta_k}{\Gamma_k} Q(x^*,y^*,z; w^k_{a g}) \\
		& \leq \frac{\theta_k}{\Gamma_k}\left\{ \frac{\eta_k}{2}\left(\left\|\hat{x}^k-x^*\right\|^2-\left\|x^{k+1}-x^*\right\|^2
		\bm{-\left(1-\alpha\right)\left\|\hat{x}^k-x^{k+1}\right\|^2}\right)\right.\\
		&\quad+\frac{\lambda_k}{2}\left(\left\|B \hat{y}^k-By^*\right\|^2-\left\|B y^{k+1}-By^*\right\|^2
		\bm{-\left(1-\beta\right)\norm{B \hat{y}^k-B y^{k+1}}^2}\right) \\
		&\quad+\frac{1}{2 \gamma_k}\left(\left\|\hat{z}^k-z\right\|^2-\left\|z^{k+1}-z\right\|^2\bm{-(\kappa-1)\norm{\hat{z}^k-z^{k+1}}^2}\right)\\
		&\quad
		\left.+\frac{\tau_k-\lambda_k}{2}\left\|A\left(x^{k+1}-x^*\right)\right\|^2 
		-\frac{\tau_k-\lambda_k}{2}\left\|B y^{k+1}-By^*\right\|^2\right\}.
		\end{aligned}
		\end{equation}
	\end{small}By Lemma \ref{fact} (i) and \eqref{GADMMhatx}, we have 
	\begin{small}
			\begin{equation*}
		\begin{aligned}
		&\quad\;\left\|\hat{x}^k-x^*\right\|^2-\left\|x^{k+1}-x^*\right\|^2-\left(1-\alpha\right)\left\|\hat{x}^k-x^{k+1}\right\|^2\\
		&=\frac{1}{2-\alpha}\left((2-\alpha)
		\left\|\hat{x}^k-x^*\right\|^2-\left\|x^{k+1}-x^*\right\|^2
		-\left(1-\alpha\right)\left\|\hat{x}^k-x^{k+1}\right\|^2\right)\\
		&=\frac{1}{2-\alpha}\left(\left\|\hat{x}^k-x^*\right\|^2-\norm{(x^{k+1}-x^*)-(1-\alpha)(\hat{x}^k-x^{k+1})}^2\right)\\
		&=\frac{1}{2-\alpha}\left(\left\|\hat{x}^k-x^*\right\|^2-\norm{(2-\alpha)x^{k+1}+(\alpha-1)\hat{x}^k-x^*}^2\right)\\
		&=\frac{1}{2-\alpha}\left(\left\|\hat{x}^k-x^*\right\|^2-\norm{\hat{x}^{k+1}-x^*}^2\right),
		\end{aligned}
		\end{equation*}
	\end{small}and from \eqref{GADMMhaty}, we can see that 
	\begin{small}
			\begin{equation*}
		\begin{aligned}
		&\quad\;\left\|B \hat{y}^k-By^*\right\|^2-\left\|B y^{k+1}-By^*\right\|^2
		-\left(1-\beta\right)\norm{B \hat{y}^k-B y^{k+1}}^2\\
		&=\frac{1}{2-\beta}\left(\norm{B \hat{y}^k-B y^*}^2-\norm{(2-\beta)By^{k+1}+(\beta-1)B\hat{y}^k-By^*}^2\right)\\
		&=\frac{1}{2-\beta}\left(\norm{B \hat{y}^k-B y^*}^2-\norm{B\hat{y}^{k+1}-By^*}^2\right).
		\end{aligned}
		\end{equation*}
	\end{small}Analogously, from Lemma \ref{fact} (ii) and \eqref{GADMMhatz}, we obtain
\begin{small}
		\begin{equation*}
	\begin{aligned}
	&\quad\;\left\|\hat{z}^k-z\right\|^2-\left\|z^{k+1}-z\right\|^2-(\kappa-1)\norm{\hat{z}^k-z^{k+1}}^2\\
	&=\frac{1}{\kappa}\left(\kappa\left(\left\|\hat{z}^k-z\right\|^2-\left\|z^{k+1}-z\right\|^2-(\kappa-1)\norm{\hat{z}^k-z^{k+1}}^2\right)\right)\\
	&=\frac{1}{\kappa}\left(\norm{\hat{z}^k-z}^2-\norm{(z^{k+1}-z)-(\kappa-1)(\hat{z}^k-z^{k+1})}^2\right)\\
	&=\frac{1}{\kappa}\left(\norm{\hat{z}^k-z}^2-\norm{(1-\kappa)\hat{z}^k+\kappa z^{k+1}-z}^2\right)\\
	&=\frac{1}{\kappa}\left(\norm{\hat{z}^k-z}^2-\norm{\hat{z}^{k+1}-z}^2\right).
	\end{aligned}
	\end{equation*}
\end{small}Using the above three equality into \eqref{GADMMeq6}, we have 
\begin{small}
		\begin{equation*}
	\begin{aligned}
	&\quad\; \frac{1}{\Gamma_k} Q(x^*,y^*,z;w^{k+1}_{ag})-\frac{1-\theta_k}{\Gamma_k} Q(x^*,y^*,z; w^k_{a g}) \\
	&\leq   \frac{\theta_k}{\Gamma_k}\left\{\frac{\eta_k}{2(2-\alpha)}\left(\left\|\hat{x}^k-x^*\right\|^2-\norm{\hat{x}^{k+1}-x^*}^2\right)
	+\frac{\lambda_k}{2(2-\beta)}\left(\norm{B \hat{y}^k-B y^*}^2-\norm{B\hat{y}^{k+1}-By^*}^2\right)\right.\\
	&\quad
	\left.+\frac{1}{2\gamma_k\kappa}\left(\norm{\hat{z}^k-z}^2-\norm{\hat{z}^{k+1}-z}^2\right)
	+\frac{\tau_k-\lambda_k}{2}\left\|A\left(x^{k+1}-x^*\right)\right\|^2 
	-\frac{\tau_k-\lambda_k}{2}\left\|B y^{k+1}-By^*\right\|^2\right\}.
	\end{aligned}
	\end{equation*}
\end{small}
	This completes the proof.
\end{proof}

Now we analysis the convergence rate of GLADMM.

\begin{theorem}\label{GADMMpa}
	In GLADMM, if the total number of iterations is set to $N$ and the parameters
	\begin{small}
			\begin{equation}\label{GADMMpara}
		\theta_k=\frac{2}{k+1}, \Gamma_k=\frac{2}{k(k+1)}, \lambda_k=\tau_k=\frac{\gamma N}{k}, \gamma_k=\frac{(2-\xi)\gamma k}{\kappa N}, \eta_k=\frac{2 L}{\alpha k},
		\xi_k=\xi\in [1.5,2).
		\end{equation}
	\end{small}
	Then, we have 
\begin{footnotesize}
		\begin{align}
	&\left|F(x^N_{ag},y^N_{ag})-F(x^*,y^*)\right|
	\leq  \frac{2L\norm{x^1-x^*}^2}{\alpha(2-\alpha)N(N-1)}+\frac{\kappa\max \left\{\left(1+\left\|z^*\right\|\right)^2, 4\left\|z^*\right\|^2\right\}}{\gamma(2-\xi)(N-1)}
	+\frac{2\gamma \norm{B y^1-B y^*}^2}{(2-\beta)(N-1)},\notag\\
	&\norm{B y^N_{ag}-A x^N_{ag}-b}\leq 
	\frac{2L\norm{x^1-x^*}^2}{\alpha(2-\alpha)N(N-1)}+\frac{\kappa\max \left\{\left(1+\left\|z^*\right\|\right)^2, 4\left\|z^*\right\|^2\right\}}{\gamma(2-\xi)(N-1)}
	+\frac{2\gamma \norm{B y^1-B y^*}^2}{(2-\beta)(N-1)}.\label{eq:feasibility}
	\end{align}
\end{footnotesize}
\end{theorem}
\begin{proof}
	Applying the inequality \eqref{GADMMres} inductively and $\frac{1-\theta_i}{\Gamma_i}=\frac{1}{\Gamma_{i-1}}$, we conclude that 
	\begin{small}
			\begin{equation*}
		\begin{aligned}
		&\quad\;\frac{1}{\Gamma_k} Q(x^*,y^*,z;w^{k+1}_{ag})\\
		&\leq \sum_{i=1}^{k} \frac{\theta_i}{\Gamma_i}\frac{\eta_i}{2(2-\alpha)}\left(\left\|\hat{x}^i-x^*\right\|^2-\norm{\hat{x}^{i+1}-x^*}^2\right)
		+\sum_{i=1}^{k}\frac{\theta_i}{\Gamma_i}\frac{1}{2\gamma_i\kappa}\left(\norm{\hat{z}^i-z}^2-\norm{\hat{z}^{i+1}-z}^2\right)\\
		&\quad + \sum_{i=1}^{k}\frac{\theta_i}{\Gamma_i}\frac{\lambda_k}{2(2-\beta)}\left(\norm{B \hat{y}^i-B y^*}^2-\norm{B\hat{y}^{i+1}-By^*}^2\right)
		-\sum_{i=1}^{k}\frac{\theta_i}{\Gamma_i}\frac{\tau_i-\lambda_i}{2}\left\|B y^{i+1}-B y^*\right\|^2\\
		&\quad +\sum_{i=1}^{k}\frac{\theta_i}{\Gamma_i}\frac{\tau_i-\lambda_i}{2}\left\|A\left(x_{i+1}-x^*\right)\right\|^2. 
		\end{aligned}
		\end{equation*}
	\end{small}Observing from \eqref{GADMMpara} that we can calculate that 
	\begin{small}
		\[
		\frac{L\theta_k}{\eta_k}\leq \alpha \in (0,1),\quad\frac{1}{\xi_k}=\frac{1}{\xi}\leq \beta\in(0,1),\quad \frac{\tau_k+(1-\xi_k)\lambda_k}{\gamma_k}\geq \kappa >1,\quad\forall k \leq N,
		\]
	\end{small}
\begin{small}
		\begin{equation*}
	\begin{aligned}
	&\begin{aligned}
	&\sum_{i=1}^{k} \frac{\theta_i}{\Gamma_i}\frac{\eta_i}{2(2-\alpha)}\left(\left\|\hat{x}^i-x^*\right\|^2-\norm{\hat{x}^{i+1}-x^*}^2\right)
	= \frac{L}{\alpha(2-\alpha)}\left(\left\|\hat{x}^1-x^*\right\|^2-\norm{\hat{x}^{k+1}-x^*}^2\right),
	\end{aligned}\\
	&	\begin{aligned}
	&\sum_{i=1}^{k}\frac{\theta_i}{\Gamma_i}\frac{1}{2\gamma_i\kappa}\left(\norm{\hat{z}^i-z}^2-\norm{\hat{z}^{i+1}-z}^2\right)
	=\frac{\kappa N}{2\gamma(2-\xi)}\left(\norm{\hat{z}^1-z}^2-\norm{\hat{z}^{k+1}-z}^2\right),
	\end{aligned}\\
	&	\begin{aligned}
	&\sum_{i=1}^{k}\frac{\theta_i}{\Gamma_i}\frac{\lambda_i}{2(2-\beta)}\left(\norm{B \hat{y}^i-B y^*}^2-\norm{B\hat{y}^{i+1}-By^*}^2\right)
	=\frac{\gamma N}{2-\beta}\left(\norm{B \hat{y}^1-B y^*}^2-\norm{B\hat{y}^{k+1}-By^*}^2\right).
	\end{aligned}
	\end{aligned}
	\end{equation*}	
\end{small}Applying the above three calculations to \eqref{GADMMres} and multiplying both sides of above inequality by $\Gamma_k$, we have
\begin{small}
		\begin{equation*}
	\begin{aligned}
	Q(x^*,y^*,z;w^{k+1}_{ag})\leq \frac{2L\norm{\hat{x}^1-x^*}^2}{\alpha(2-\alpha)k(k+1)}+\frac{\kappa N\norm{\hat{z}^1-z}^2}{\gamma(2-\xi)k(k+1)}
	+\frac{2\gamma N\norm{B \hat{y}^1-B y^*}^2}{(2-\beta)k(k+1)}.
	\end{aligned}
	\end{equation*}
\end{small}Letting $k=N-1$, noting $\hat{x}^1=x^1$, $\hat{y}^1=y^1$, $\hat{z}^1=0$ and applying Lemmas \ref{rate1} and \ref{rate2} with $\rho=\max \left\{1+\left\|z^*\right\|, 2\left\|z^*\right\|\right\}$, we obtain
\begin{small}
		\begin{equation*}
	\begin{aligned}
	&\left|F(x^N_{ag},y^N_{ag})-F(x^*,y^*)\right|
	\leq  \frac{2L\norm{x^1-x^*}^2}{\alpha(2-\alpha)N(N-1)}+\frac{\kappa\max \left\{\left(1+\left\|z^*\right\|\right)^2, 4\left\|z^*\right\|^2\right\}}{\gamma(2-\xi)(N-1)}
	+\frac{2\gamma \norm{B y^1-B y^*}^2}{(2-\beta)(N-1)},\\
	&\norm{B y^N_{ag}-A x^N_{ag}-b}\leq 
	\frac{2L\norm{x^1-x^*}^2}{\alpha(2-\alpha)N(N-1)}+\frac{\kappa\max \left\{\left(1+\left\|z^*\right\|\right)^2, 4\left\|z^*\right\|^2\right\}}{\gamma(2-\xi)(N-1)}
	+\frac{2\gamma \norm{B y^1-B y^*}^2}{(2-\beta)(N-1)}.
	\end{aligned}
	\end{equation*}
\end{small}This completes the proof.
\end{proof}
\begin{remark}\label{rem:4}
	Similar to GPGM and GLADMM proposed in Section 2 and 3, We use G\"uler-type acceleration technique to accelerate L-ADMM. Specifically,  we set $\alpha_k,\beta_k \in (0,1)$ and $\kappa_k\in (1,2)$ in formulation \eqref{GADMMeq6} from the proof of Lemma \ref{l:GLADMM}, which is an important lemma in convergence analysis of GLADMM. Then, we can utilize the negative terms $\bm{-\|x^k-\hat{x}^{k-1}\|^2}$, $\bm{-\|y^k-\hat{y}^{k-1}\|^2}$, and $\bm{-\|z^k-\hat{z}^{k-1}\|^2}$ to design the extrapolation steps \eqref{GADMMhatx}, \eqref{GADMMhaty} and \eqref{GADMMhatz}. The proposed GLADMM can accelerate both the primal and dual variables and achieve a better convergence rate as in AL-ADMM\cite{ouyang2015accelerated}.
\end{remark}

\begin{remark}\label{rem:42}
	We compare  the convergence result \cite[Formulation (2.46)]{ouyang2015accelerated} of AL-ADMM in \cite[Theorem 2.9]{ouyang2015accelerated} with  \eqref{eq:feasibility} in Theorem \ref{GADMMpa}. Although the total convergence rate of both AL-ADMM  and GLADMM in \eqref{eq:feasibility} is $O(\frac{1}{N})$,  the first term of \cite[Formulation (2.46)]{ouyang2015accelerated} has a convergence rate of $O(\frac{1}{N^{3/2}})$, while the first term of \eqref{eq:feasibility} has a convergence  rate of $O(\frac{1}{N^2})$. This indicates that the proposed GLADMM achieves a better convergence rate compared to the AL-ADMM in \cite{ouyang2015accelerated}, see Table \ref{Table 0}.
\end{remark}

\section{Numerical experiment}
In this section, we present three numerical experiments to demonstrate the effectiveness of G\"uler-type acceleration technique for GPG, GLALM and GLADMM. We mainly consider the following three different types of problems: $\ell_1$ regularized logistic regression problem, quadratic programming and compressive sensing. All numerical experiments are run in MATLAB R2016b on a PC with an Intel(R) Xeon(R) Silver 4210R CPU @2.40GHz and 64 GB of RAM under the Windows 10 operating system.

The aim of the numerical experiments is to demonstrate how G\"uler-type acceleration techniques can be applied to gradient-based algorithms, using PGM, LALM, and L-ADMM as examples, to improve their performance. Specifically, we compare Nesterov's second APGM \cite{tseng2010approximation}, ALALM  \cite{xu2017accelerated}, and AL-ADMM  \cite{ouyang2015accelerated} based on Nesterov's extrapolation techniques, with the performance of GPGM, GLALM, and GLADMM algorithms obtained by incorporating G\"uler-type acceleration techniques. The G\"uler-type acceleration techniques can be further extended to gradient-based acceleration algorithms that utilize Nesterov's extrapolation techniques.

\subsection{$\ell_1$ regularized logistic regression}
In this subsection, we focus on the $\ell_1$ regularized logistic regression problem introduced by \cite{wen2017linear}:
\begin{equation}\label{lg}
\min _{\tilde{x} \in \mathbb{R}^n, x_0 \in \mathbb{R}} F(x)=\sum_{i=1}^m \log \left(1+\exp \left(-b_i\left(a_i^{\top} \tilde{x}+x_0\right)\right)\right)+\lambda\|\tilde{x}\|_1,
\end{equation}
where $a_i \in \mathbb{R}^n$, $b_i \in\{-1,1\}$, $i=1,2, \ldots, m$, with $m<n$, and regularization parameter $\lambda>0$. In our algorithms below we take $L=0.25 \lambda_{\max }\left(D^{\top} D\right)$, and $D$ is the matrix whose $i$th row is given by $(a_i^\top 1)$.
The setting of the problem is the same as in \cite[Section 4.1]{wen2017linear}, and the termination criterions are 
\begin{equation*}
\max \left\{\frac{\left|f\left(x^k\right)+g\left(x^k\right)-G\left(u^k\right)\right|}{\max \left\{f\left(x^k\right)+g\left(x^k\right), 1\right\}}, \frac{50\left|e^T u^k\right|}{\max \left\{\left\|u^k\right\|, 1\right\}}\right\} \leq 10^{-8}
\end{equation*}
or if the maximum number of iterations reaches 2000, where the above formulation is  dual feasibility violation (see \cite[Section 4.1]{wen2017linear} for details). 

In this experiment, we generated three instances of the $\ell_1$ regularized logistic regression problems with different parameters $(m,n,s)=(300,3000,30)$, $(500,5000,50)$, and $(800,8000,80)$. The problem generation follows the guidelines outlined in \cite[Section 4.1]{wen2017linear}. 
For each problem instance, we apply Nesterov's accelerated projected gradient (with $\alpha=1$ in GPGM) and GPGM (with $\alpha=0.8$) to solve \eqref{lg}. 
Note that we first solve the problem \eqref{lg} to high accuracy by calling MATLAB's intrinsic functions, and use the resulting high-precision optimal value and solution as \(F^*\) and \(x^*\).

\begin{figure}[htbp]
	\centering
	\subfigure[]{
		\includegraphics[scale=0.3]{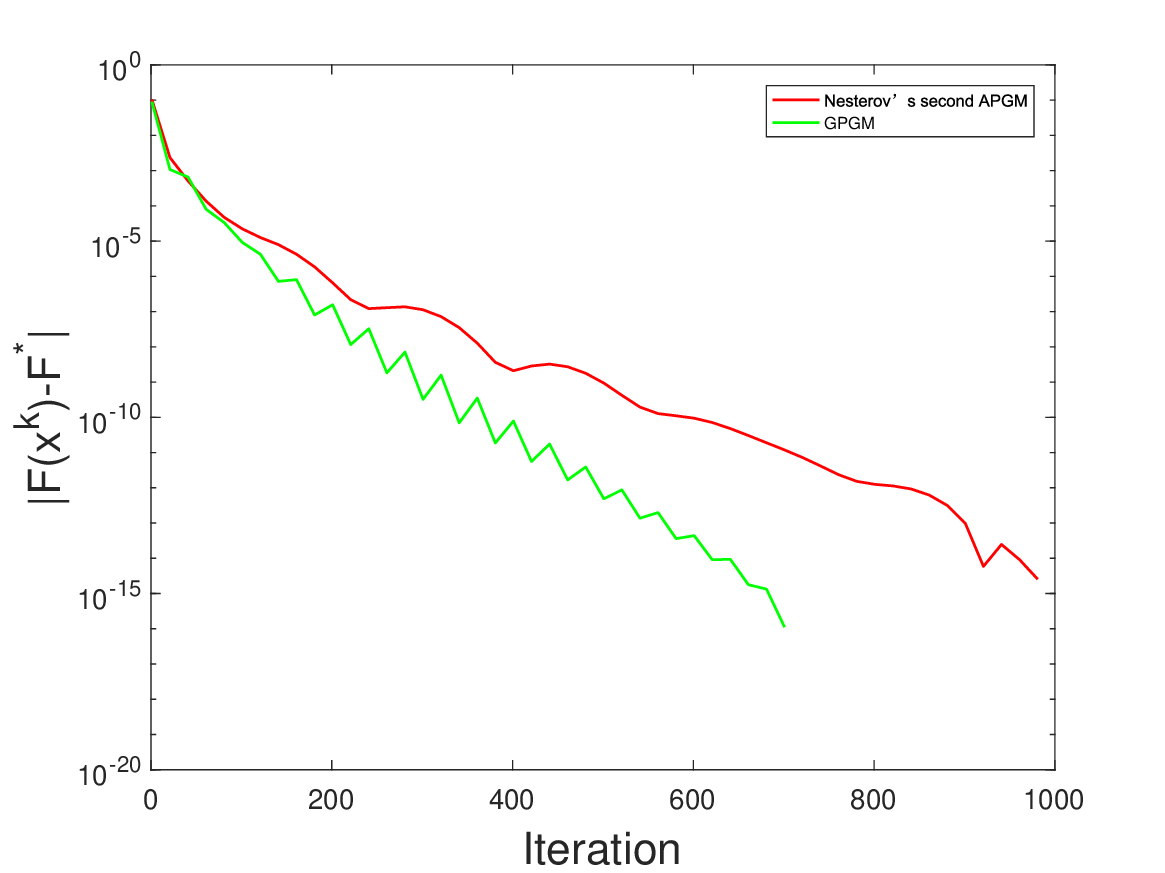}
	}
	\subfigure[]{
		\includegraphics[scale=0.3]{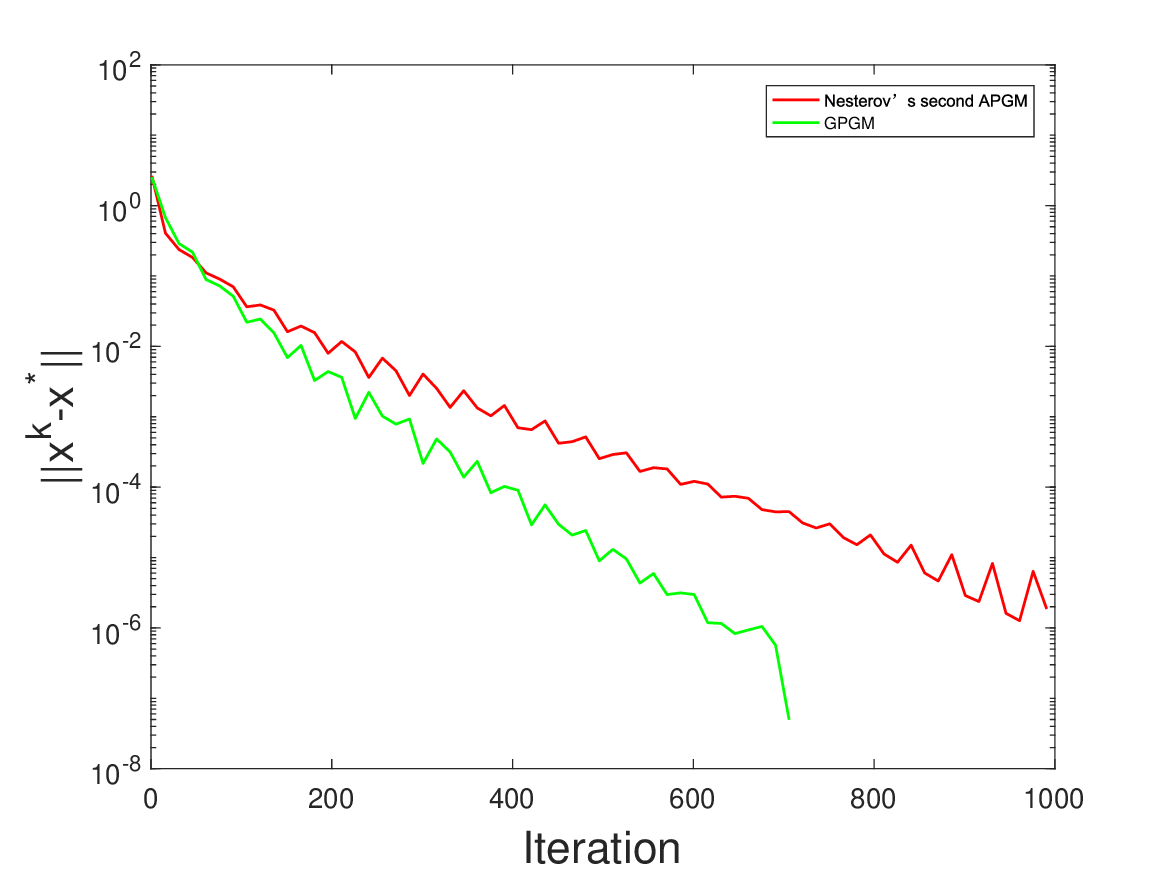}
	}
	\caption{$\ell_1-logistic$ $m=300, n=3000, s=30$}
	\label{Fig1}
		\subfigure[]{
		\includegraphics[scale=0.3]{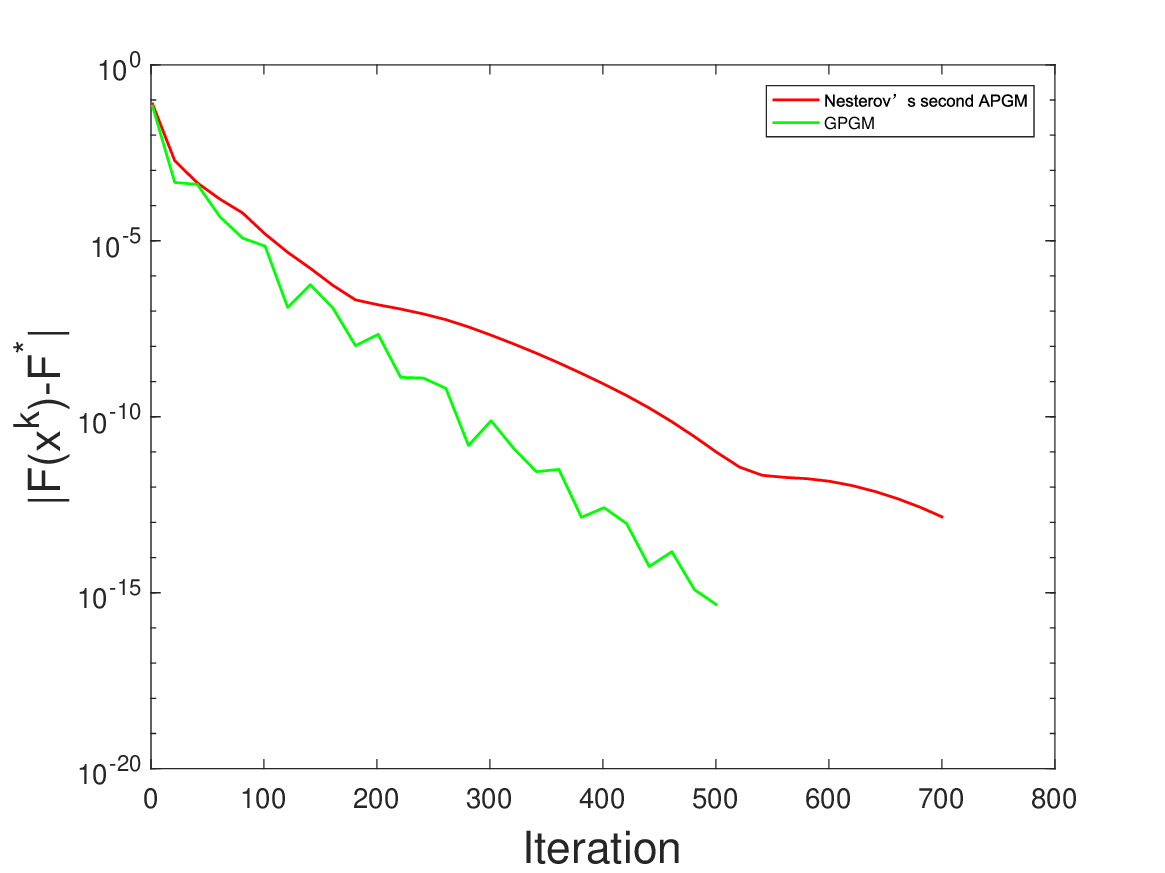}
	}
	\subfigure[]{
		\includegraphics[scale=0.3]{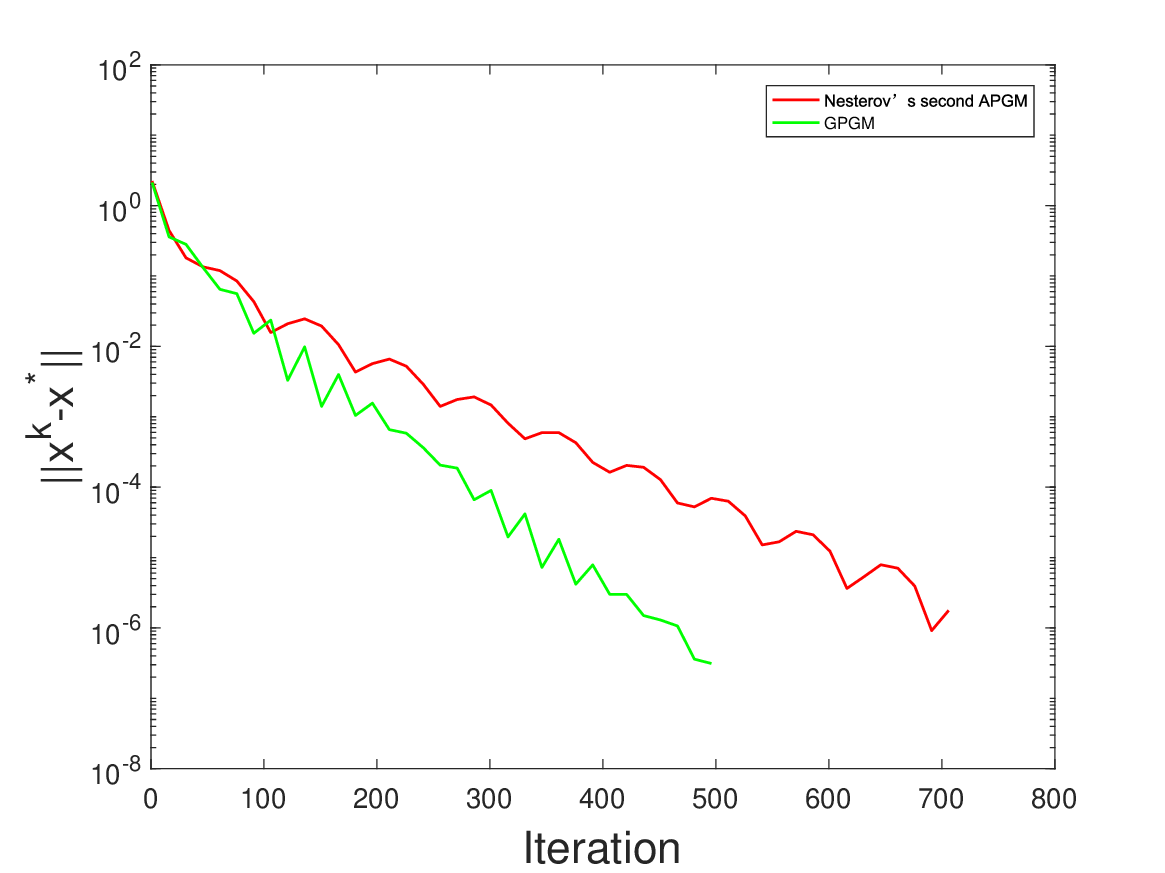}
	}
	\caption{$\ell_1-logistic$ $m=500, n=5000, s=50$}
	\label{Fig2}
	\subfigure[]{
		\includegraphics[scale=0.3]{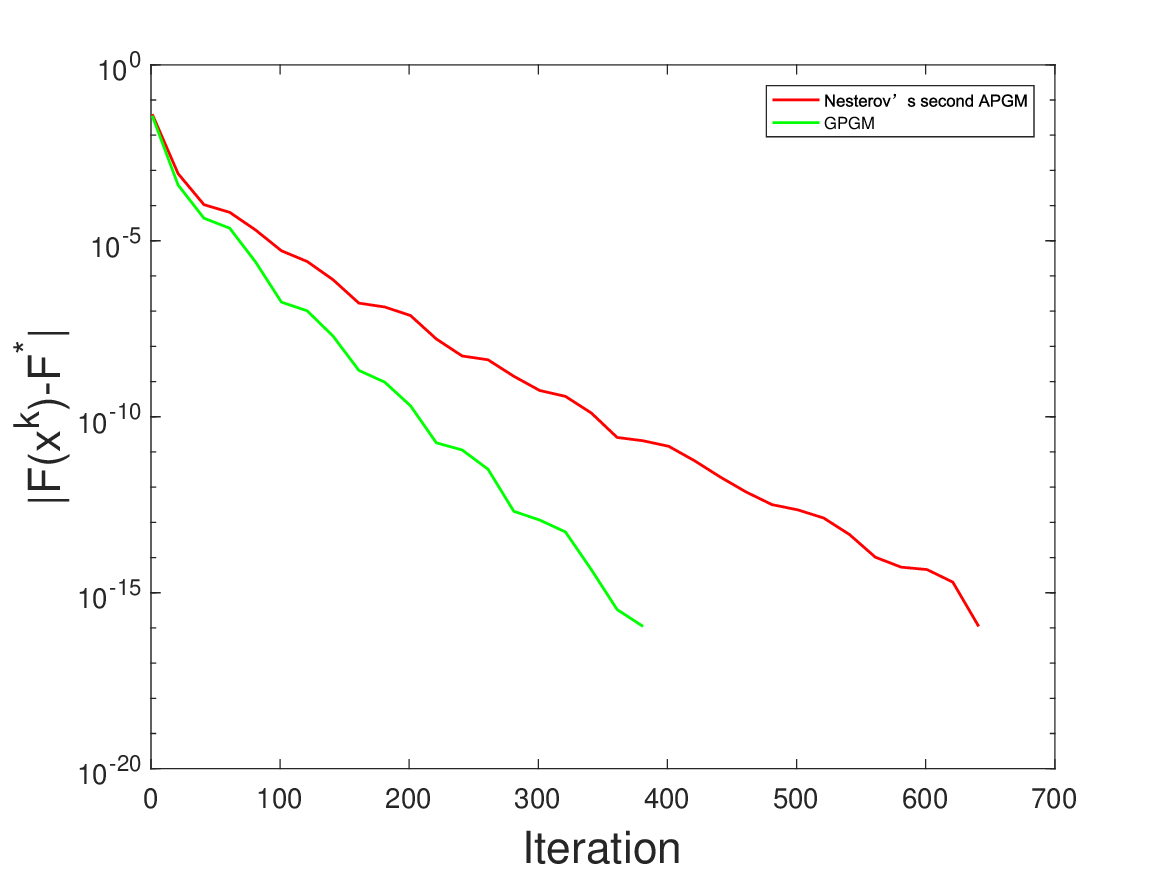}
	}
	\subfigure[]{
		\includegraphics[scale=0.3]{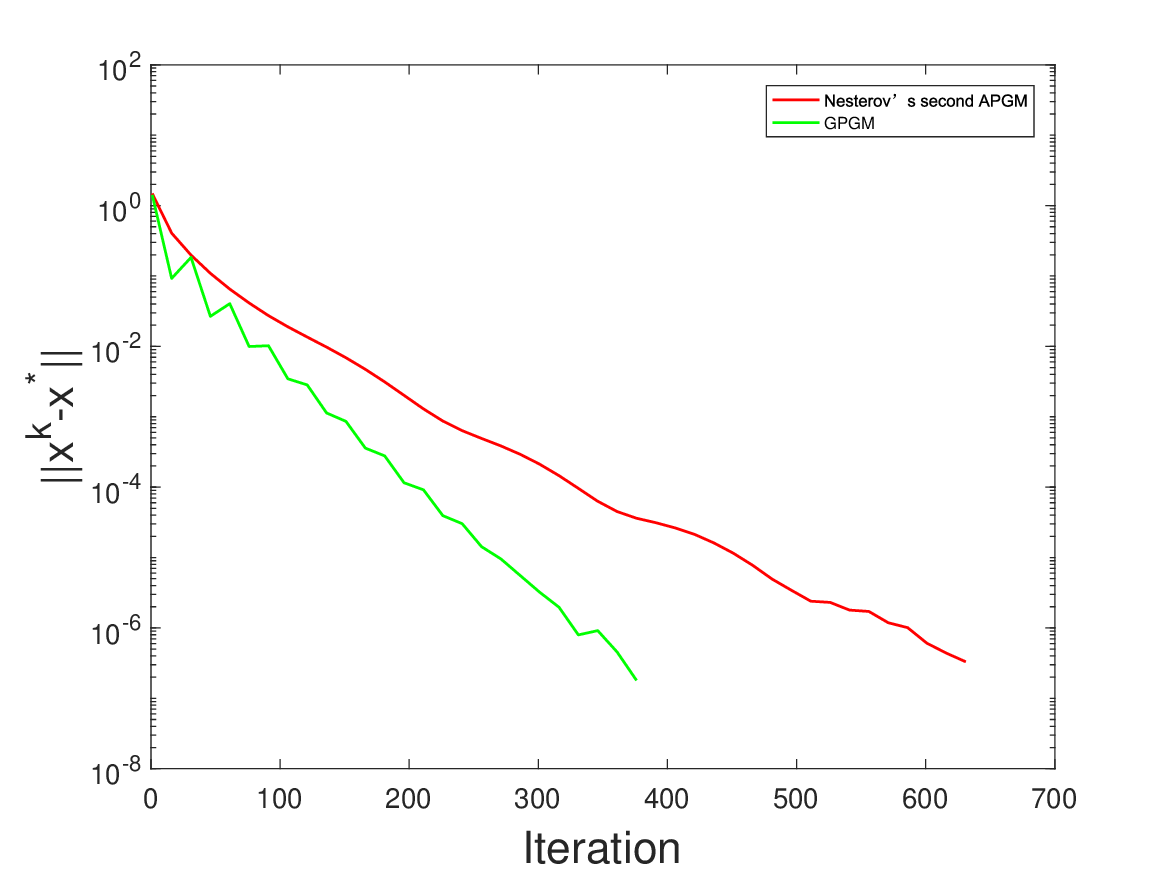}
	}
	\caption{$\ell_1-logistic$ $m=800, n=8000, s=80$}
	\label{Fig3}
\end{figure}


The results are presented in Figures \ref{Fig1}-\ref{Fig3}. In part (a) of each figure, the plot shows the relationship between $F(x^k)-F^*$ and the number of iterations. In part (b) of each figure, the plot illustrates the relationship between $\norm{x^k-x^*}$ and the number of iterations. 
These figures  demonstrate that GPGM outperforms Nesterov's accelerated projected gradient. 

\subsection{Quadratic programming}

In this subsection, we consider the following nonnegative linearly constrained quadratic programming:
\begin{equation}\label{eq:qp}
\min _x F(x)=\frac{1}{2} x^{\top} Q x+c^{\top} x \text {, s.t. } A x=b, x\geq 0, 
\end{equation}
where $Q\in \bbr^{n\times n}$ is a positive semidefinite matrix, $A\in \bbr^{m\times n}$, $c\in \bbr^n$, and $b\in \bbr^m$. Our objective is to demonstrate the effectiveness of our algorithm. In the algorithms below, we take $L=\lambda_{\max }\left(Q^{\top} Q\right)$
and terminate the algorithms when the maximum number of iterations reaches 1000.

We now conduct numerical experiments to compare our Algorithm \ref{GALALM} with the ALALM with adaptive parameters \cite[Algorithm 1]{xu2017accelerated}. We generated three instances of the quadratic programming problem \eqref{eq:qp} for different values of $(m,n)=(80,1000)$, $(100,1000)$, and $(120,1500)$. The matrices $Q\in \bbr^{n\times n}$, $A\in \bbr^{m\times n}$, and the vectors $c\in\bbr^n$ and $b\in\bbr^m$ were randomly generated with independent and identically distributed (i.i.d.) standard Gaussian entries. For our GLALM algorithm, we choose the parameters as specified in Theorem \ref{ThGAL} with $\alpha=0.5$, $\kappa=1.5$, $\gamma=15m$, and $\eta=2|Q|_2$. Similar to problem \eqref{lg}, the high-precision optimal value and solution as \(F^*\) and \(x^*\) through MATLAB's built-in functions. 

Figures \ref{Fig24}-\ref{Fig26}
show the objective distance to the optimal value $\left|F(x)-F\left(x^*\right)\right|$ and the feasibility violation $\|A x-b\|$ for both GLALM and ALALM. It is evident that GLALM outperforms ALALM significantly in terms of both objective and feasibility measures. 
This further confirms the effectiveness of leveraging negative terms to accelerate the algorithm.

\begin{figure}[htbp]
	\centering
	\subfigure[]{
		\includegraphics[scale=0.3]{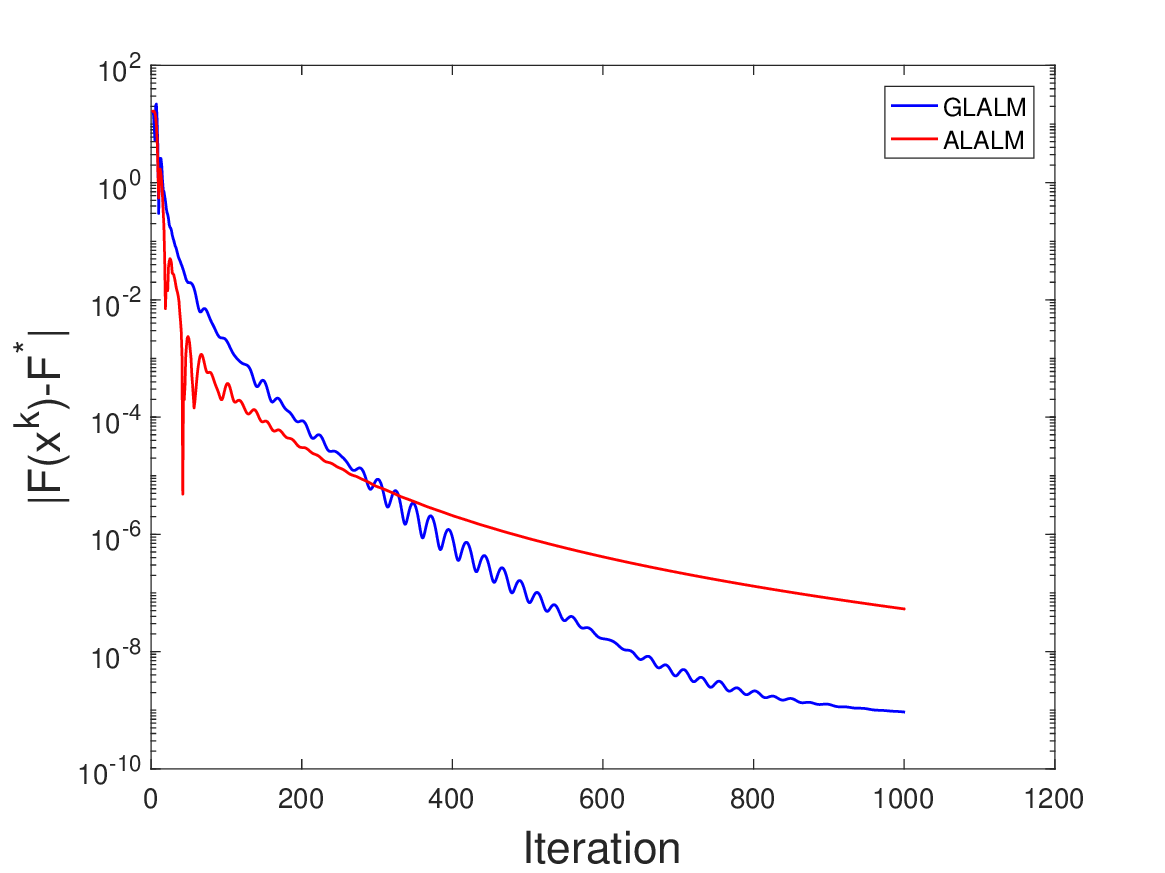}
	}
	\subfigure[]{
		\includegraphics[scale=0.3]{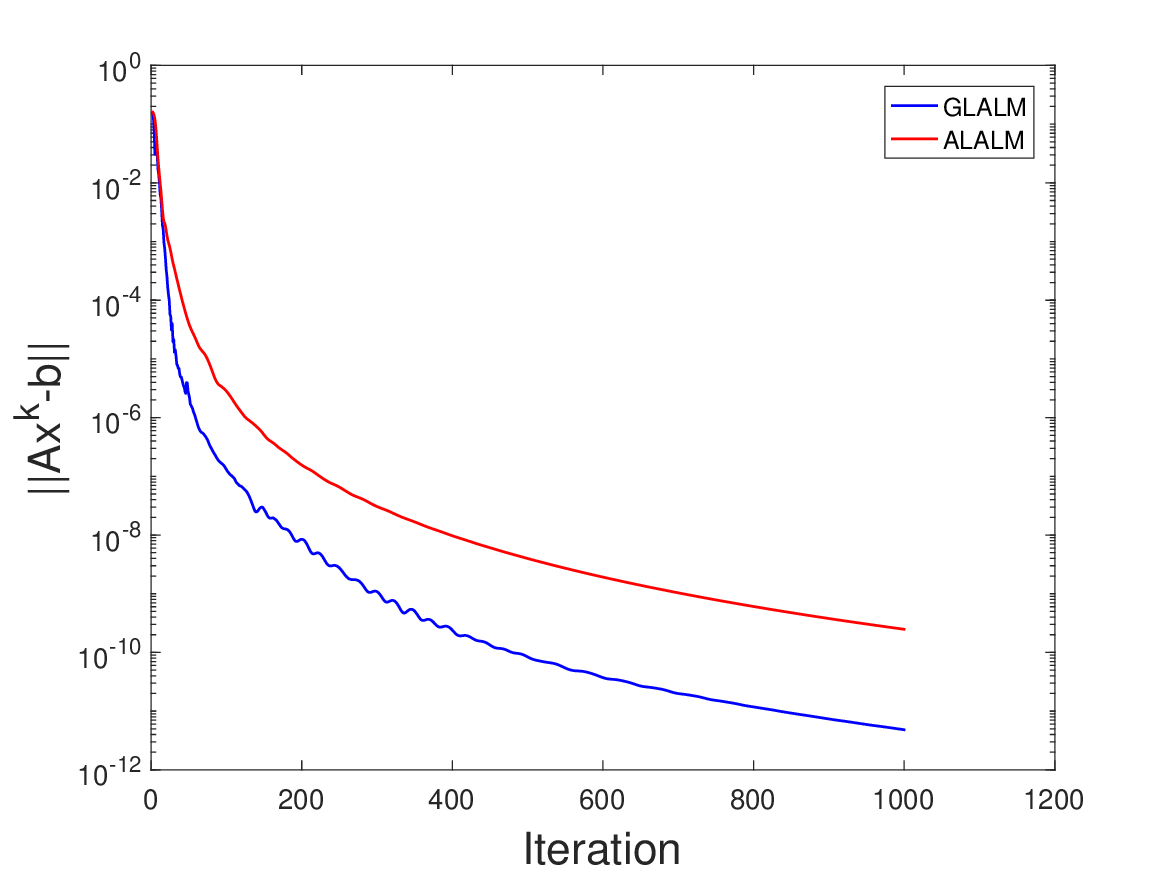}
	}
	\caption{quadratic programming $m=80, n=1000$}
	\label{Fig24}
		\subfigure[]{
			\includegraphics[scale=0.3]{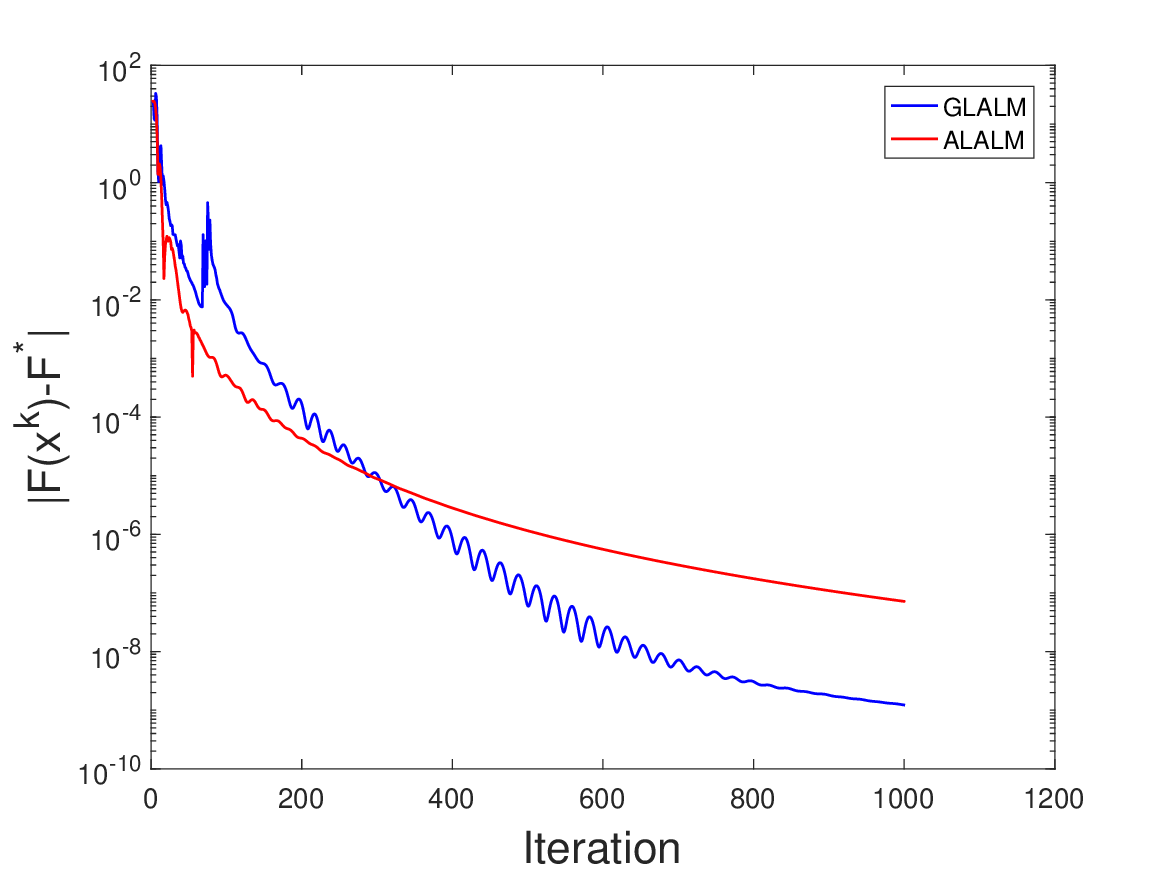}
		}
		\subfigure[]{
			\includegraphics[scale=0.3]{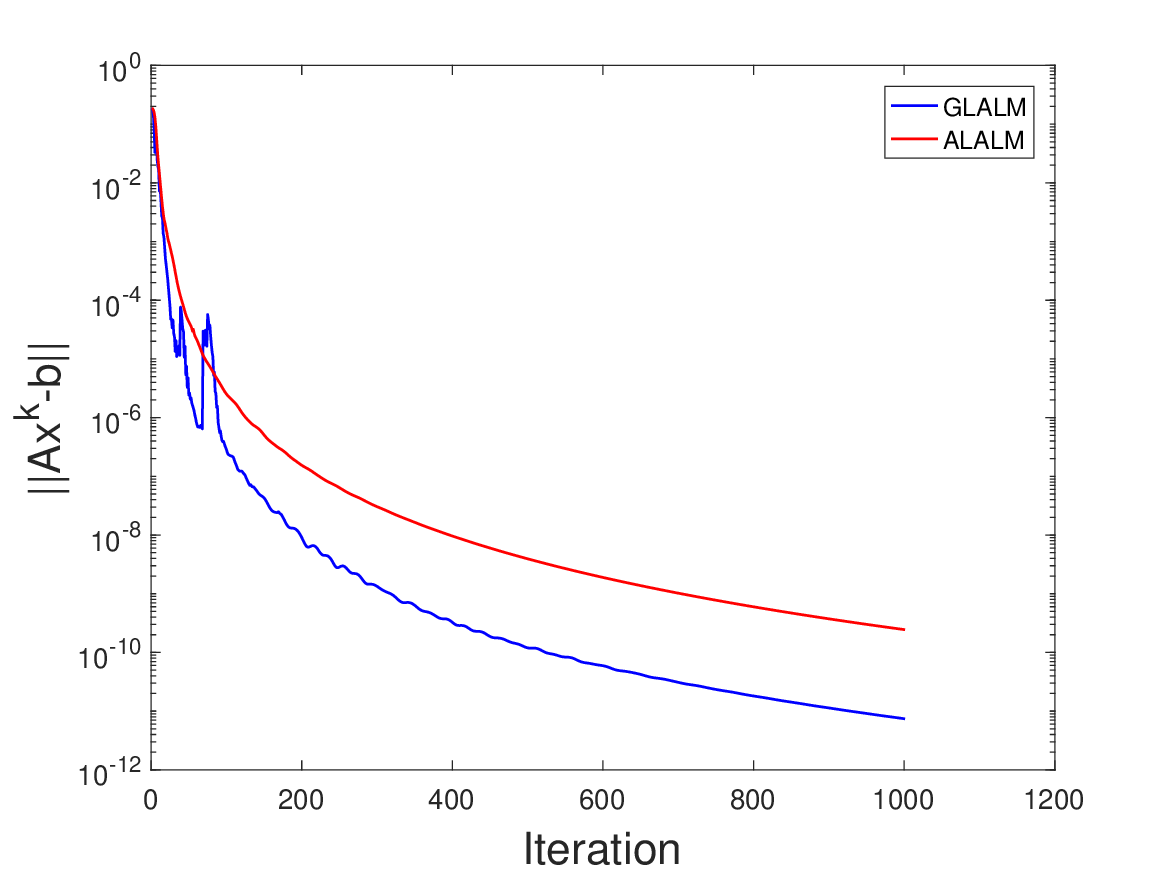}
		}
		\caption{quadratic programming $n=1000, m=100$}
		\label{Fig25}
		\subfigure[]{
			\includegraphics[scale=0.3]{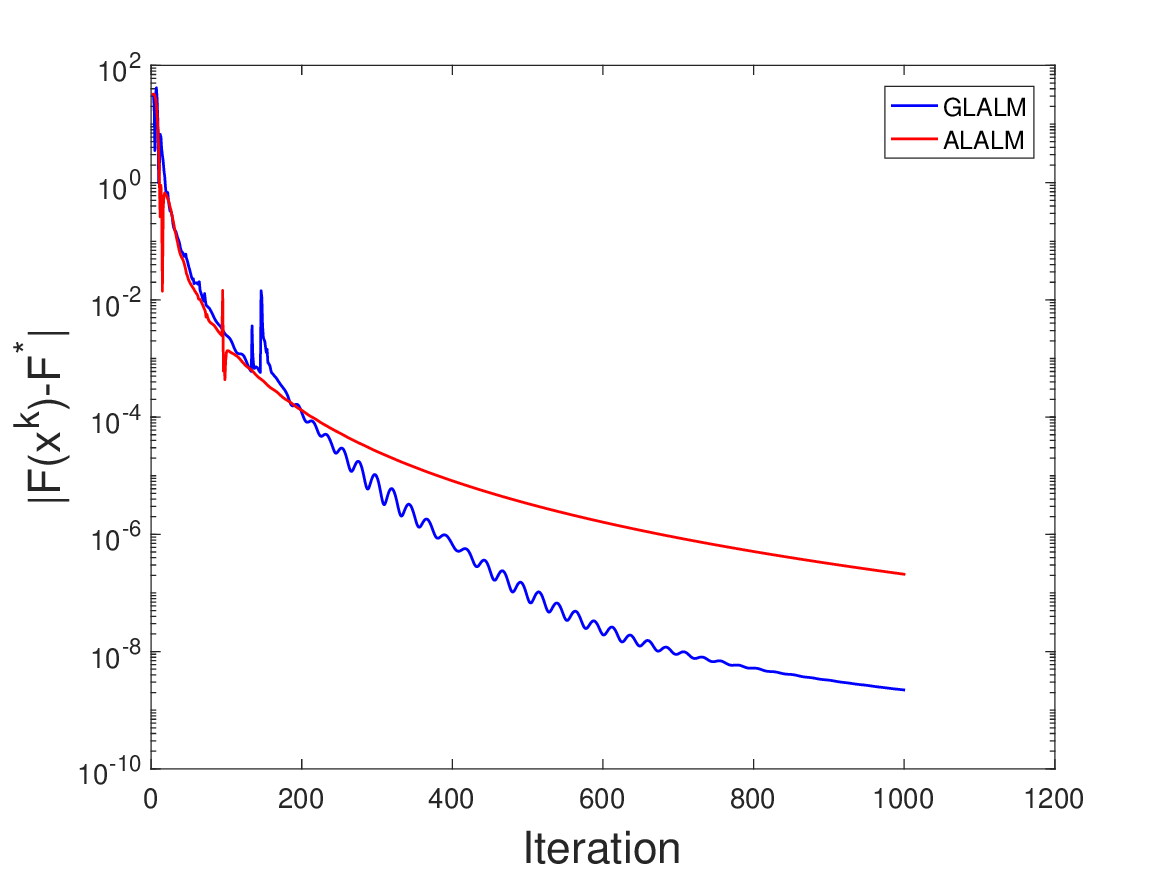}
		}
		\subfigure[]{
			\includegraphics[scale=0.3]{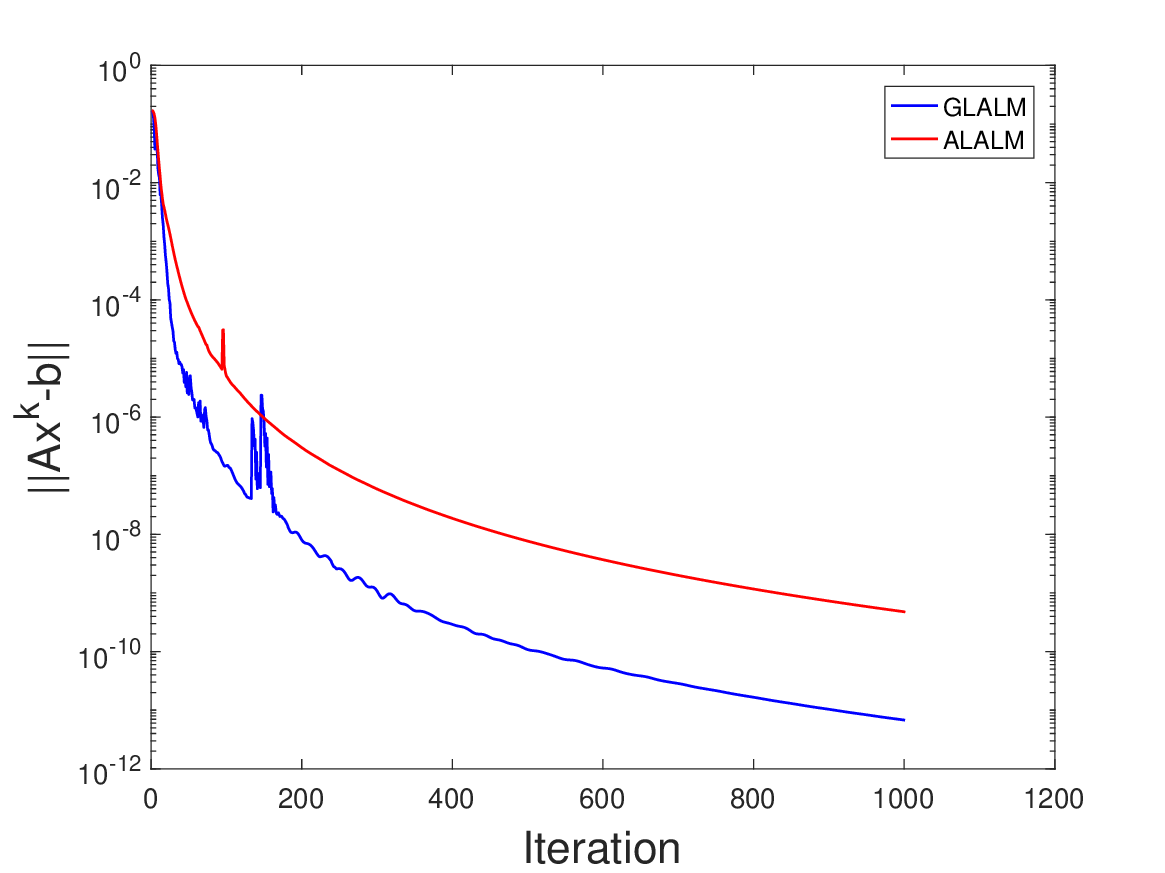}
		}
		\caption{quadratic programming $n=1500, m=120$}
		\label{Fig26}
\end{figure}


\subsection{Compressive sensing}
In this subsection, we present the experimental results comparing L-ADMM, AL-ADMM, and GLADMM for solving an image reconstruction problem. The problem is formulated as follows:\begin{equation}\label{CS}
\min _{x \in \bbr^n} \frac{1}{2}\|D x-b\|^2+\lambda\|A x\|_{2,1},
\end{equation}
where $x$ represents a two-dimensional image to be reconstructed, $\|A x\|_{2,1}$ denotes the discrete total variation (TV) seminorm, $D$ is the given acquisition matrix based on the data acquisition physics, and $b$ corresponds to the observed data. We can reformulate \eqref{CS} into the form of \eqref{AECCO} with the following optimization problem:
\begin{equation*}
\begin{aligned}
& \min _{x \in \bbr^n,y\in \bbr^m} \frac{1}{2}\|D x-b\|^2+\lambda\|y\|_{2,1},
&\text{s.t.}\; y-Ax=0.
\end{aligned}
\end{equation*}

In this experiment, we consider the instance where the acquisition matrix $D \in \bbr^{m \times n}$ is generated independently from a normal distribution $N(0,1 / \sqrt{m})$. This type of acquisition matrix is commonly used in compressive sensing (see, e.g., \cite{baraniuk2008simple}). For a given $D$, the measurements $b$ are generated as $b=D x_{\text {true }}+\varepsilon$, where $x_{\text {true }}$ is a 64 by 64 Shepp-Logan phantom \cite{shepp1974fourier} with intensities in the range $[0,1]$ (so $n=4096$) and $\varepsilon \equiv N\left(0,0.001 I_n\right)$. We choose $m=1229$ to achieve a compression ratio of approximately $30 \%$, and set $\lambda=10^{-3}$ in \eqref{CS}. We apply L-ADMM, AL-ADMM, and GLADMM to solve \eqref{CS}. For this case, we use the parameters in \cite[Theorem 2.9]{ouyang2015accelerated} and  Theorem \ref{GADMMpa} with $N=300$ for AL-ADMM and GLADMM, respectively. To ensure a fair comparison, we use the same constants $L=\lambda_{\max }\left(D^T D\right)$ and $\gamma=1 /|A|$ for all algorithms. We report both the primal objective function value and the relative reconstruction error defined as
$
r(\tilde{x}):=\frac{\left\|\tilde{x}-x_{\text {true }}\right\|}{\left\|x_{\text {true }}\right\|} .
$
Figure \ref{ADMMFig7} shows that GLADMM  outperforms L-ADMM and AL-ADMM in solving \eqref{CS}, which is consistent with our theoretical results in Theorem \ref{GADMMpa}.

\begin{figure}[htbp]
	\centering
	\subfigure[]{
		\includegraphics[scale=0.3]{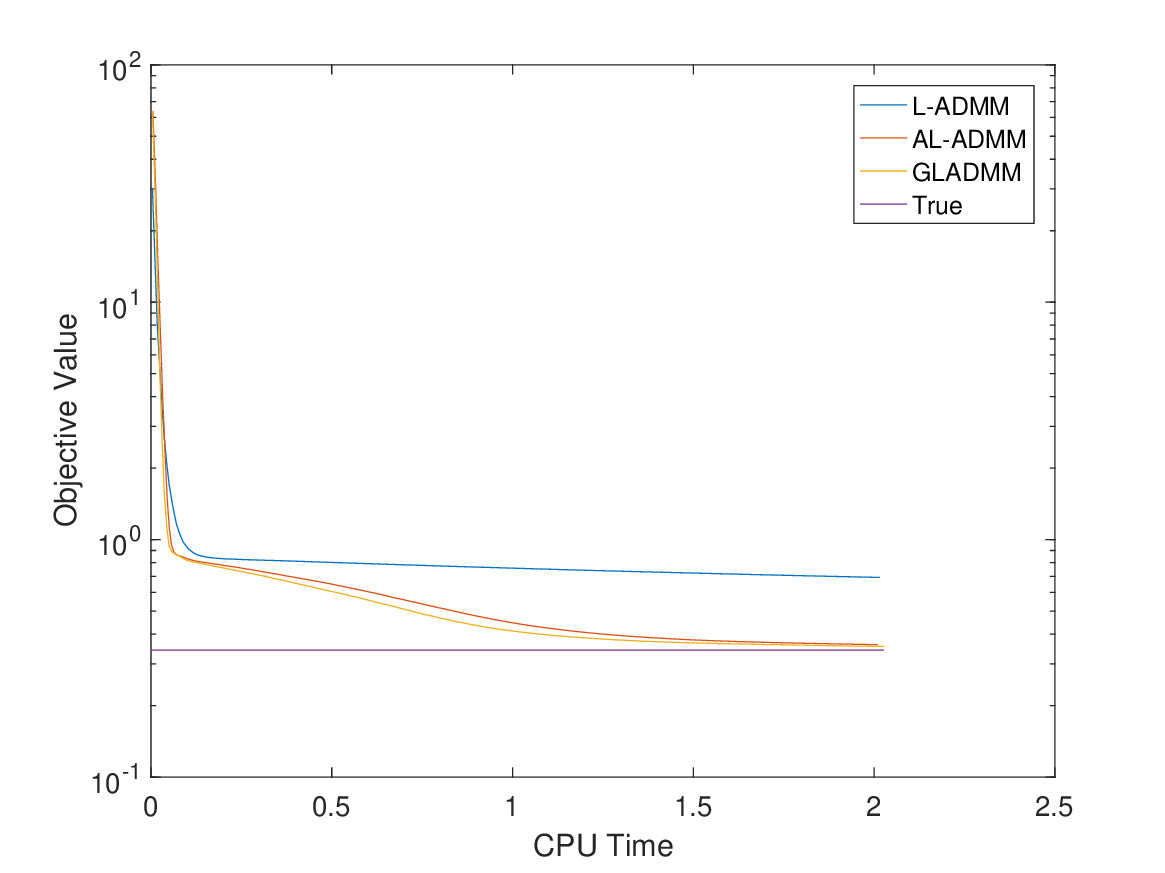}
	}
	\subfigure[]{
		\includegraphics[scale=0.3]{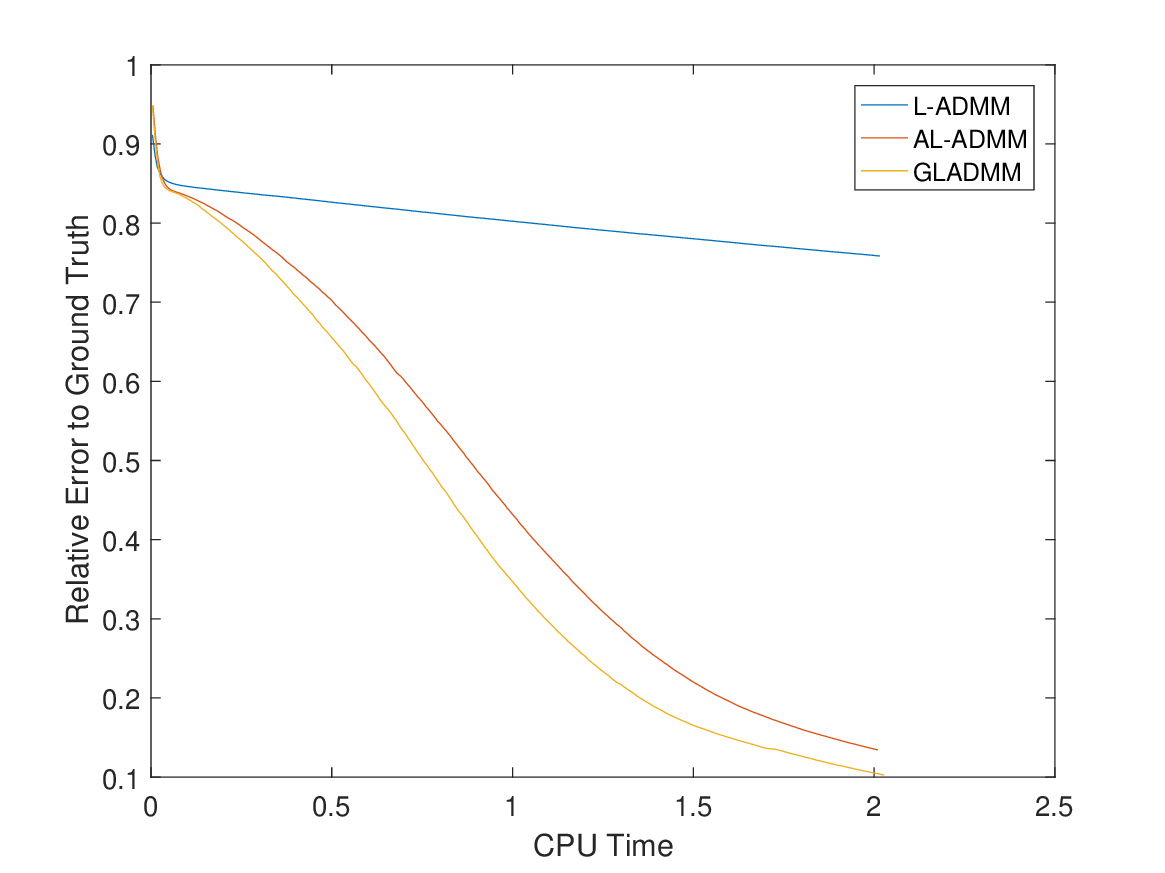}
	}
	\caption{Comparisons of 300 iterations of L-ADMM, AL-ADMM and GLADMM}
	\label{ADMMFig7}
\end{figure}

\section{Conclusion}
In this paper, we propose three acceleration algorithms, namely GPGM (Algorithm 1), GLALM (Algorithm 2), and GLADMM (Algorithm 3), to solve three structured non-smooth problems. These algorithms utilize G\"uler-type acceleration techniques, while GLALM and GLADMM also accelerate both primal and dual variables. The key idea, as described in Remarks \ref{rem:2}, \ref{rem:3} and \ref{rem:4}, is to leverage negative terms for achieving larger descent by appropriately configuring the parameters. Compared to the existing algorithms, GPGM and GLALM achieve the same convergence rate of $O(\frac{1}{k^2})$ but with a broader range of coefficient options. On the other hand, GLADMM maintains the same total convergence rate of $O(\frac{1}{N})$ as in  existing results, but with an improved partial convergence rate from $O(\frac{1}{N^{3/2}})$ to $O(\frac{1}{N^2})$ (Remark \ref{rem:42}).  The G\"uler-type acceleration techniques can be extended to other gradient-based algorithms, enhancing their acceleration capabilities. Numerical experiments demonstrate the efficiency of these improved accelerated algorithms compared to existing methods. The corresponding stochastic versions of these algorithms have broad applications in fields such as statistics, machine learning, and data mining.

The idea of Güler-type acceleration techniques is to enhance acceleration methods by incorporating the negative term $-\norm{x^k-\hat{x}^{k-1}}^2$ in the extrapolation step. In this paper, we specifically apply Güler-type acceleration techniques to three algorithms to illustrate their advantages. However, this technique can be extended to a wider range of gradient-based algorithms. Additionally, further research is needed to investigate whether this technique can be effectively applied to accelerate non-convex gradient-based algorithms.

\section*{Acknowledgments}
	We extend our gratitude to Prof. Guanghui Lan, Prof. Yuyuan Ouyang, and Prof. Bo Wen for generously providing us with valuable source codes.

\section*{Declarations}
{\bf Funding:} This research is supported by the Natural Science Research Start-up Foundation of Recruiting Talents of Nanjing University of Posts and Telecommunications (Grant NY225056), the National Natural Science Foundation of China (Grant 12331011), and the Key Laboratory of NSLSCS, Ministry of Education.

\noindent {\bf Data availability:} Data will be made available on reasonable request.

\noindent {\bf Conflict of interest:} The authors declare that they have no conflict of interest.

	\bibliographystyle{unsrt}
	\bibliography{reference}
\end{document}